\documentclass[]{scrartcl}

\usepackage[utf8]{luainputenc}
\usepackage[USenglish]{babel}
\usepackage{csquotes}

\usepackage[a4paper,top=27mm,bottom=20mm,inner=25mm,outer=20mm]{geometry}

\usepackage[%
  backend=bibtex,bibencoding=ascii,
  style=numeric-comp,
  giveninits=true, uniquename=init, 
  natbib=true,
  url=true,
  doi=true,
  isbn=false,
  backref=false,
  maxnames=99,
  ]{biblatex}
\usepackage{amsmath}
\allowdisplaybreaks
\numberwithin{equation}{section}
\usepackage{amssymb}
\usepackage{commath}
\usepackage{mathtools}
\usepackage{bbm}
\usepackage{nicefrac}
\usepackage{subdepth}

\usepackage{siunitx}
\usepackage{algorithm}
\usepackage{algpseudocode}

\usepackage{amsthm}
\usepackage{thmtools}
\usepackage{etoolbox}
\makeatletter
\patchcmd{\thmt@setheadstyle}
 {\bgroup\thmt@space}
 {\thmt@space}
 {}{}
\patchcmd{\thmt@setheadstyle}
 {\egroup\fi}
 {\fi}
 {}{}
\makeatother
\declaretheoremstyle[
  bodyfont=\normalfont\itshape,
  headformat=\NAME\ \NUMBER\NOTE,
]{myplain}
\declaretheoremstyle[
  headformat=\NAME\ \NUMBER\NOTE,
]{mydefinition}
\newcommand{\envqed}{{\lower-0.3ex\hbox{$\triangleleft$}}}
\declaretheorem[style=myplain,numberwithin=section]{theorem}
\declaretheorem[style=myplain,numberlike=theorem]{lemma}

\declaretheorem[style=mydefinition,numberlike=theorem,qed=\envqed]{definition}
\declaretheorem[style=mydefinition,numberlike=theorem,qed=\envqed]{remark}
\declaretheorem[style=mydefinition,numberlike=theorem,qed=\envqed]{example}

\usepackage[plainpages=false,pdfpagelabels,hidelinks,unicode]{hyperref}

\usepackage{color}
\usepackage{graphicx}
\usepackage[small]{caption}
\usepackage{subcaption}

\begingroup\expandafter\expandafter\expandafter\endgroup
\expandafter\ifx\csname pdfsuppresswarningpagegroup\endcsname\relax
\else
\fi

\usepackage{booktabs}
\usepackage{rotating}
\usepackage{multirow}

\usepackage{enumitem}

\usepackage{ifluatex}
\ifluatex
  \usepackage[no-math]{fontspec}
\else
  \usepackage[T1]{fontenc}
\fi
\usepackage{newpxtext,newpxmath}

\usepackage{bm}
\usepackage{cleveref}

\newcommand{\tr}{\textrm}
\newcommand{\mr}{\mathrm}

\newcommand{\R}{\mathbb{R}}

\newcommand{\cU}{\mathcal{U}}
\newcommand{\cM}{\mathcal{M}}
\newcommand{\cI}{\mathcal{I}}
\newcommand{\cV}{\mathcal{V}}
\newcommand{\cE}{\mathcal{E}}

\DeclareMathOperator{\sech}{sech}
\DeclareMathOperator{\diag}{diag}

\newcommand{\orcid}[1]{ORCID:~\href{https://orcid.org/#1}{#1}}
\usepackage{authblk}

\newenvironment{keywords}{\par\textbf{Key words.}}{\par}
\newenvironment{AMS}{\par\textbf{AMS subject classification.}}{\par}

\title{Structure-Preserving Numerical Methods for Two Nonlinear Systems of Dispersive Wave Equations}

\author[1]{Joshua~Lampert\thanks{\orcid{0009-0007-0971-6709}}}
\affil[1]{Department of Mathematics, University of Hamburg, Germany}

\author[2]{Hendrik~Ranocha\thanks{\orcid{0000-0002-3456-2277}}}
\affil[2]{Institute of Mathematics, Johannes Gutenberg University Mainz, Germany}

\date{September 8, 2025}

\makeatletter
\hypersetup{pdfauthor={Joshua Lampert, Hendrik Ranocha}}
\hypersetup{pdftitle={Structure-Preserving Numerical Methods for Two Nonlinear Systems of Dispersive Wave Equations}}
\makeatother

\begin{document}

\maketitle

\begin{abstract}
\noindent
  We use the general framework of summation-by-parts operators
to construct conservative, energy-stable, and well-balanced
semidiscretizations of two different nonlinear systems of
dispersive shallow water equations with varying bathymetry:
(i) a variant of the coupled Benjamin-Bona-Mahony (BBM) equations
and (ii) a recently proposed model by Svärd and Kalisch (2025) with
enhanced dispersive behavior. Both models share the property of
being conservative in terms of a nonlinear invariant, often
interpreted as energy. This property is preserved exactly
in our novel semidiscretizations. To obtain fully-discrete
energy-stable schemes, we employ the relaxation method.
Our novel methods generalize energy-conserving
methods for the BBM-BBM system to variable bathymetries. Compared to the
low-order, energy-dissipative finite volume method proposed by Svärd and Kalisch,
our schemes are arbitrary high-order accurate, energy-conservative or -stable,
can deal with periodic and reflecting boundary conditions, and can be any
method within the framework of summation-by-parts operators including
finite difference and finite element schemes.
We present improved numerical properties of our methods in some test cases.

\end{abstract}



\section{Introduction}
\label{sect:introduction}

We develop numerical methods for nonlinear dispersive wave equations that are based on the shallow water equations (SWEs). The SWEs in one spatial dimension are usually written in the form of a hyperbolic balance law \cite{leveque2002finite}
\begin{subequations}
	\begin{align}
		h_t + (hv)_x &= 0,\label{eq:swe-mass}\\
		(hv)_t + \left(\frac{1}{2}gh^2 + hv^2\right)_x &= -ghb_x,\label{eq:swe-momentum}
	\end{align}
\end{subequations}
where $h > 0$ is the water height, $v$ the velocity, and $g > 0$ denotes the gravitational acceleration. The topography of the bottom (bathymetry) is denoted as $b(x)$, see also \Cref{fig:bathymetry} for a sketch.
The different notations that are used in some of the main sources are
summarized in \Cref{table:notation}.

\begin{figure}[!h]
	\centering
	\includegraphics[width=0.55\textwidth]{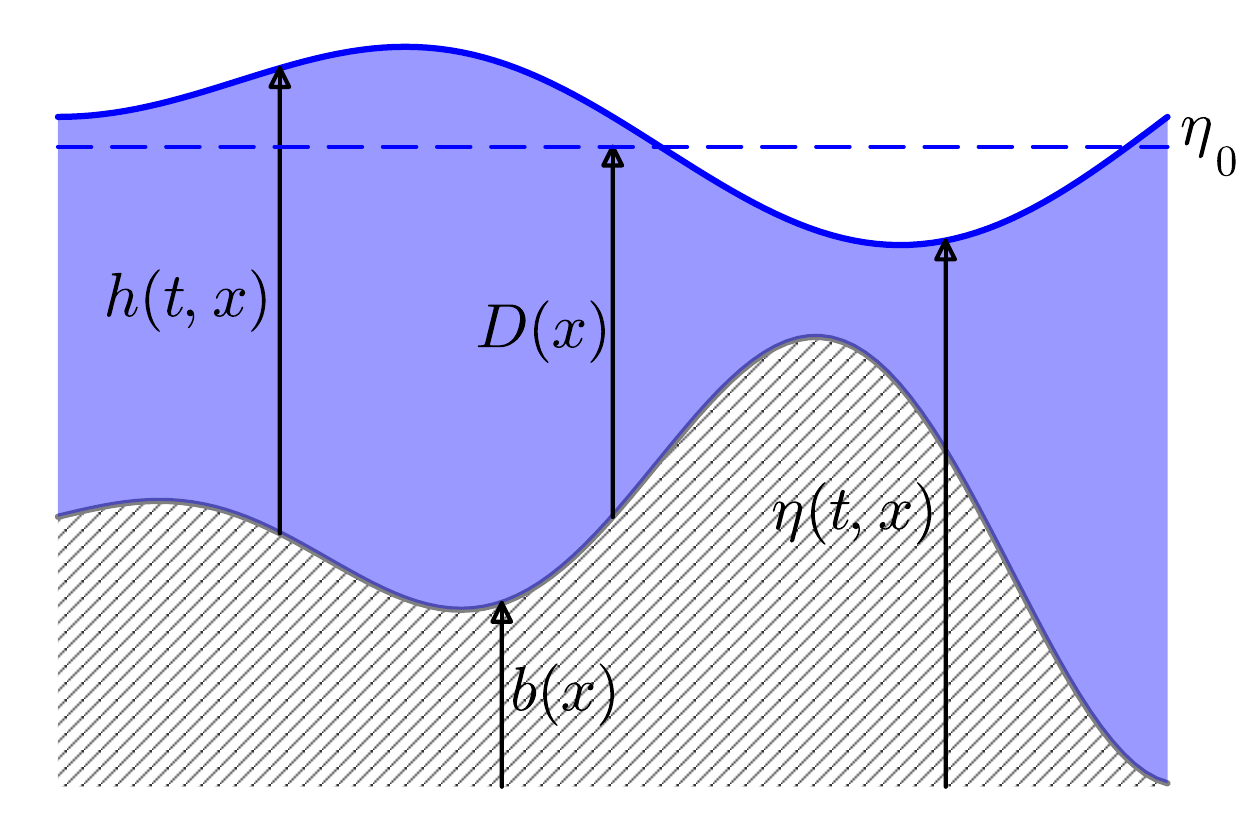}
	\caption{Water height $h$, reference total water height $\eta_0$, and bathymetry $b$. The still water depth is $D = \eta_0 - b$ and the total water height is $\eta = h + b$}
	\label{fig:bathymetry}
\end{figure}

\begin{table}[!htbp]
	\centering
	\caption{Notation used for (dispersive) shallow water models}
	\label{table:notation}
	\begin{tabular}{ m{2.2cm}m{2.4cm}m{2.2cm}m{2.4cm}m{2.3cm}m{2.2cm}m{2.2cm} }
		\toprule
		& this article & Israwi et al.\ (2021) \cite{israwi2021regularized} & Mitsotakis et al.\ (2021) \cite{mitsotakis2021conservative} & Ranocha et al.\ (2021) \cite{ranocha2021broad} & Svärd et al.\ (2025) \cite{svard2025novel}\\
		\midrule\midrule
		water height & $h = \eta + D - \eta_0$ & $\eta + D$ & $\eta + D$ & $\eta + 1$ & $d$\\
		\midrule
		total water height & $\eta = h + b$ & $\eta$ & $\eta$ & $\eta$ & $d + b$\\
		\midrule
		still water height & $\eta_0$ & $0$ & $0$ & $0$ & $H$\\
		\midrule
		still water depth & $D$ & $D$ & $D$ & $1$ & $h$\\
		\midrule
		bathymetry & $b = \eta_0 - D$ & $-D$ & $-D$ & $-1$ & $b = H - h$\\
		\midrule
		velocity & $v$ & $\bm u$ & $u$ & $u$ & $v$\\
		\midrule
		discharge & $P = hv$ & -- & -- & -- & $P = dv$\\
		\bottomrule
	\end{tabular}
\end{table}

The SWEs are widely used to describe phenomena with large wavelengths and comparably small water heights, such as tsunamis generated by (bigger) earthquakes or inundations. However, if we consider waves with shorter wavelengths in the deep sea, the effect of dispersion becomes pertinent \cite{glimsdal2013dispersion}. This concerns, e.g., waves generated by submarine landslides. Yet, dispersion is not modeled by the SWEs. Therefore, alternative models that are able to describe dispersive behavior need to be constructed. On the other hand, we would like to keep at least two properties of the SWEs. First, the special nonlinearity of the SWEs has proven to describe the evolution of oceanic water waves well.
Second, the dispersive model should preserve the physically interpretable energy/entropy (in)equality. Thus, we focus on models that build on the SWEs, but add additional terms that are capable to describe dispersion, while satisfying an entropy condition. Specifically, we look at two different models: A system of two coupled Benjamin-Bona-Mahony (BBM) equations with variable bottom topography \cite{israwi2021regularized} called BBM-BBM equations and a dispersive system recently proposed by Svärd and Kalisch \cite{svard2025novel}.

There are several other dispersive shallow water models, see \cite[Chapter 6]{kim2014finite} for a brief overview. Widely used dispersive equations are the models by Peregrine \cite{peregrine1967long}, by Madsen and S{\o}rensen \cite{madsen1992newform}, and the related model of Schäffer and Madsen \cite{schaeffer1995further}, see also \cite{kim2017boussinesq,berger2021towards}. A further important model that is closely related to the BBM-BBM equations is the Nwogu system \cite{nwogu1993alternative}. All of these models can be seen as optimization of a previous one, and share a similar structure: The dispersive terms are modeled as higher-order derivative terms.

All models mentioned above do not conserve the mathematical entropy function conserved by the SWEs \cite{israwi2021regularized}.
However, both the BBM-BBM model and the equations of Svärd and Kalisch mimic the entropic structure of the classical SWEs.
Thus, we concentrate on these two models to develop structure-preserving numerical methods,
which typically leads to physically more accurate numerical solutions
and can also improve stability properties \cite{ranocha2020hamiltonian,ranocha2020general}.

A different approach to the above-mentioned models is taken in \cite{escalante2020general}, where a dispersive shallow water model is formulated as a hyperbolic system of equations containing only first derivatives. Nonetheless, the authors of \cite{escalante2020general} show that the models by Peregrine and Madsen-S{\o}rensen can be formulated with their model.
There are also other dispersive shallow water models with an entropic structure,
e.g., the Serre-Green-Naghdi equations and hyperbolic approximations thereof
\cite{favrie2017rapid,busto2021high,guermond2022hyperbolic,ranocha2024structure} or the model studied in \cite{aissiouene2020two,bonnetbendhia2021pseudo}.
However, we concentrate on the BBM-BBM equations and the model of Svärd and Kalisch
in this article and leave the other models for future work.

To construct numerical methods that are entropy-conservative on a semidiscrete level (discrete in space, continuous in time), we utilize the framework of summation-by-parts (SBP) operators and split forms \cite{fisher2013discretely}.
This allows a systematic approach to obtain entropy-conservative and entropy-dissipative semidiscretizations.
The idea is to mimic the mathematical properties used on the continuous level (SBP operators mimic integration by parts, split forms mimic the product and chain rule) at the semidiscrete level.
To integrate the resulting ordinary differential equations in time, we employ relaxation Runge-Kutta (RRK) methods that preserve the entropic structure also for the fully-discrete solution \cite{ketcheson2019relaxation,ranocha2020relaxation}.

We focus on models with variable bottom topography (bathymetry). In this situation, the SWEs and their dispersive generalizations have
non-trivial steady solutions. By using consistent SBP operators, our numerical scheme will automatically preserve the most important
set of stationary solutions, i.e., the lake-at-rest solution. Thus, our numerical methods are well-balanced
in the fully wet case. We only consider one-dimensional models without wave breaking mechanisms explicitly included.

In the following \Cref{sect:dispersive-swe-models}, we introduce the governing
equations, starting with the classical SWEs and continuing with two dispersive
modifications. \Cref{sect:numerical-methods} starts with a brief introduction
to SBP operators to make this article sufficiently self-contained. The main
theoretical contributions of this work are Sections~\ref{sect:num-BBM-BBM} and
\ref{sect:num-svard-kalisch}, where we develop and analyze novel
structure-preserving numerical methods. Compared to existing methods
for the BBM-BBM system \cite{ranocha2021broad,mitsotakis2021conservative},
we generalize energy-conserving methods to variable bathymetries in
\Cref{sect:num-BBM-BBM}. In contrast to the low-order entropy-dissipative
methods of \cite{svard2025novel}, our novel schemes developed for the
Svärd-Kalisch model in \Cref{sect:num-svard-kalisch} are arbitrarily
high-order accurate and can be applied not only to periodic but also to
reflecting boundary conditions.
We present numerical results in \Cref{sect:results} to verify the theoretical
results. Finally, we summarize our work in \Cref{sect:conclusions}.

\section{Dispersive shallow water models}\label{sect:dispersive-swe-models}

First, we briefly review results about the SWEs and two systems of nonlinear dispersive wave
equations. The SWEs \eqref{eq:swe-mass}--\eqref{eq:swe-momentum} are often used to model water waves with a long wavelength compared to the water
depth and serve as a baseline model for this work. In this article, we consider one space dimension $x\in\R$
and periodic or reflecting boundary conditions.

The water height and the momentum (or discharge) $P = hv$ form the vector of conserved variables $\bm u = (h, P)^T: \R^+\times\R\to\cU$, where $\cU = \left\{(u_1, u_2)^T\in\R^2\,|\,u_1 > 0\right\}$. Written as a balance law
\begin{equation}\label{eq:balance-law}
	\bm u_t + \bm f(\bm u)_x = -\bm g(\bm u, z)z_x =: \bm s(\bm u, z),
\end{equation}
the flux function $\bm f$ and the source term $\bm s$ for the SWEs are
\begin{equation*}
	\bm f(\bm u) = \begin{pmatrix}hv\\\frac{1}{2}gh^2 + hv^2\end{pmatrix} = \begin{pmatrix}P\\\frac{1}{2}gh^2 + \frac{P^2}{h}\end{pmatrix},\quad \bm g(\bm u, z) = \begin{pmatrix}0\\h\end{pmatrix}, \quad z(x) = gb(x).
\end{equation*}

For smooth solutions $\bm u$, the SWEs can also be written in primitive variables $\bm q = (\eta, v)^T$ consisting of the total water height $\eta = h + b$ and the velocity $v$:
\begin{subequations}
	\begin{align}
		\eta_t + ((\eta + D - \eta_0)v)_x &= 0,\label{eq:swe-prim1}\\
		v_t + g\eta_x + vv_x &= 0.\label{eq:swe-prim2}
	\end{align}
\end{subequations}
In the primitive formulation, we usually work with the \emph{still water depth} $D = \eta_0 - b$ for a reference water height $\eta_0$. From \eqref{eq:swe-mass}, we can see that for some spatial domain $\Omega\subset\R$ with periodic boundary conditions, the total mass $\cM$ and total discharge $\cI$ for constant bathymetry are conserved over time, where
\begin{equation*}
	\cM(t; \bm u) = \int_\Omega h\tr dx,\quad \cI(t; \bm u) = \int_\Omega P\tr dx.
\end{equation*}
For reflecting boundary
conditions ($\eta_x = 0$ and $v = 0$ at the boundaries), the total mass is still conserved, but the total
discharge not anymore.

In many cases, balance laws have an additional conserved, or at least dissipated, nonlinear
functional. Following \cite[Chapter 3.1]{bouchut2004nonlinear},
a twice differentiable function $U(\bm u, z)$ is a mathematical \emph{(partial) entropy} for a balance law \eqref{eq:balance-law}
if there is a \emph{(partial) entropy flux} $F(\bm u, z)$, such that
\begin{equation*}
	U_u(\bm u, z)\cdot\bm f^\prime(\bm u) = F_u(\bm u, z) \quad\forall \bm u\in\cU,
\end{equation*}
where the index $u$ denotes the transposed gradient with respect to $\bm u$.
The mathematical entropy pair $(U, F)$ satisfies for smooth solutions $\bm u$ the balance law
\begin{equation}\label{eq:entropy-equality-balance-law}
	U(\bm u, z)_t + F(\bm u, z)_x = -mz_x
\end{equation}
for $m(\bm u, z) = \bm w(\bm u, z)\cdot\bm g(\bm u, z) - F_z(\bm u, z)$. Here, $\bm w(\bm u, z) = \nabla_uU(\bm u, z)$ are the \emph{entropy variables}. For general solutions we postulate an entropy inequality
\begin{equation*}
	U(\bm u, z)_t + F(\bm u, z)_x \le -mz_x.
\end{equation*}
For the SWEs with bathymetry as source term we have the partial entropy
\begin{equation}\label{eq:entropy-swe}
	U(\bm u, z) = \frac{1}{2}(hv^2 + gh^2) + ghb = \frac{1}{2}\left(\frac{P^2}{h} + gh^2\right) + hz
\end{equation}
and the partial entropy flux
\begin{equation}\label{eq:entropy-flux-swe}
	F(\bm u, z) = \frac{1}{2}hv^3 + gh^2v + gbhv = \frac{1}{2}\frac{P^3}{h^2} + ghP + Pz.
\end{equation}
Since $m = 0$, \eqref{eq:entropy-equality-balance-law} yields the conservation of the total entropy
\begin{equation*}
	\cE(t; \bm u, z) = \int_\Omega U(\bm u, z)\tr dx.
\end{equation*}
For the SWEs, the mathematical entropy $U$ is the physical total (i.e., potential plus kinetic) energy. In primitive variables, the energy is usually written as
\begin{equation}\label{eq:entropy-swe-prim}
	U(\bm q, z) = \frac{1}{2}(g\eta^2 + (\eta + D - \eta_0)v^2),
\end{equation}
which differs from the energy \eqref{eq:entropy-swe} in the constant term $\frac{1}{2}gb^2$. This term is neglected since we are usually only interested in the change of energy. In the following, we will drop the explicit dependency of the energy function on the bathymetry and write, e.g., $U(\bm u)$ instead of $U(\bm u, z)$.

The SWEs have non-trivial steady-state solutions, i.e., solutions that are constant in time \cite[p. 67]{bouchut2004nonlinear}. We focus on the \emph{lake-at-rest} solution
\begin{equation*}
	v = 0, \qquad
	(h + b)_x = 0.
\end{equation*}
These special solutions of the SWEs are important since many applications are concerned
with small perturbations thereof. Moreover, very often a realistic simulation
converges to a steady-state for $t\to\infty$ \cite{lukacova2006well}.

\subsection{BBM-BBM equations with variable topography}
The BBM-BBM equations have been studied, e.g., in \cite{bona1998boussinesq}.
As common in the literature \cite{katsaounis2020boussinesq,mitsotakis2021conservative,israwi2023equations},
we let $\eta_0 = 0$ for simplicity. For constant bathymetry, the BBM-BBM equations can be written as
\begin{subequations}
	\begin{align}
		\eta_t + ((\eta + D)v)_x -\frac{1}{6}D^2\eta_{xxt}&= 0,\label{eq:BBMBBM-const1}\\
		v_t + g\eta_x + vv_x -\frac{1}{6}D^2v_{xxt}&= 0.\label{eq:BBMBBM-const2}
	\end{align}
\end{subequations}
Compared to the SWEs in primitive variables \eqref{eq:swe-prim1}--\eqref{eq:swe-prim2}, the BBM-BBM equations differ in the additional dispersive terms containing mixed third derivatives. Well-posedness results for the BBM-BBM equations are obtained in \cite{bona1998boussinesq} and later more generally in \cite{bona2004boussinesq}. An analytical soliton solution with constant speed $c =  \frac{5}{2}\sqrt{gD}$, parameter $\rho = \frac{18}{5}$, and some $x_0\in\R$ has been derived in \cite{chen1998exact} for the dimensionless case and reads for  $\theta = \frac{1}{2}\sqrt{\rho}\frac{\xi - x_0}{D}$, $\xi = x - ct$ as
\begin{subequations}
	\begin{align}
		\eta(t, x) &= \frac{15}{4}D\left(2\sech^2(\theta) - 3\sech^4(\theta)\right),\label{eq:BBMBBM-soliton1}\\
		v(t, x) &= \frac{15}{2}\sqrt{gD}\sech^2(\theta).\label{eq:BBMBBM-soliton2}
	\end{align}
\end{subequations}
Some works extend the classical BBM-BBM equations to a varying bathymetry, e.g., \cite{mitsotakis2009boussinesq,katsaounis2020boussinesq,israwi2021regularized,israwi2023equations}. Here, we focus on the model
\begin{subequations}
	\begin{align}
		\eta_t + ((\eta + D)v)_x -\frac{1}{6}(D^2\eta_{xt})_x&= 0,\label{eq:BBMBBM1}\\
		v_t + g\eta_x + vv_x -\frac{1}{6}(D^2v_{t})_{xx}&= 0\label{eq:BBMBBM2}
	\end{align}
\end{subequations}
analyzed in two-dimensional form in \cite{israwi2021regularized,israwi2023equations}. The equations can also be seen as a simplification of the model from \cite{katsaounis2020boussinesq,antonopoulos2025bona}. They are not only derived from classical asymptotic reasonings as approximations of the Euler equations considering long waves with small amplitude, but are also justified by variational principles \cite{israwi2021regularized}.
Clearly, \eqref{eq:BBMBBM1}--\eqref{eq:BBMBBM2} turn into \eqref{eq:BBMBBM-const1}--\eqref{eq:BBMBBM-const2} if the bottom topography is constant. Therefore, we also refer to \eqref{eq:BBMBBM1}--\eqref{eq:BBMBBM2} as BBM-BBM equations. In \cite{israwi2021regularized} it is shown that the two-dimensional BBM-BBM equations with mildly varying topography are well-posed for slip wall boundary conditions.

Defining the operators $F_1 := (\eta + D)v - \frac{1}{6}D^2\eta_{xt}$ and $F_2 := g\eta + \frac{1}{2}v^2 - \frac{1}{6}(D^2v_t)_x$, the BBM-BBM equations can be written in conservation form:
\begin{equation*}
	\eta_t + (F_1)_x = 0, \qquad
	v_t + (F_2)_x = 0.
\end{equation*}
Thus, the total mass $\cM$ and the total velocity $\cV$
\begin{equation*}
	\cM(t; \eta) = \int_\Omega\eta\tr dx, \quad \cV(t; v) = \int_\Omega v\tr dx
\end{equation*}
are invariants of the BBM-BBM equations with periodic boundary conditions.
In particular, it is not the total momentum (for constant bathymetry) that is conserved, but the total velocity.
Contrary to many other dispersive shallow water models, the BBM-BBM equations preserve the same energy
\begin{equation}\label{eq:primitive-total-entropy}
	\cE(t; \bm q) = \frac{1}{2}\int_\Omega g\eta^2 + (\eta + D)v^2\tr dx
\end{equation}
as the SWEs \cite[p. 757]{israwi2021regularized}. This is a useful feature for numerical approximations of the BBM-BBM equations because energy-conserving discretizations can stabilize the schemes and lead to better approximations.

To study the dispersive behavior of the BBM-BBM equations, we take a look at the linear dispersion relation. Assuming a flat bathymetry, linearizing the equations around $v = 0 + v^\prime$, $\eta = h_0 + h^\prime$,
and omitting nonlinear terms, we arrive at the linearized equations
\begin{equation*}
	h^\prime_t + h_0v^\prime_x - \frac{1}{6}h_0^2h^\prime_{xxt} = 0, \qquad
	v^\prime_t + gh^\prime_x - \frac{1}{6}h_0^2v^\prime_{xxt} = 0.
\end{equation*}
Taking one time derivative of the first and one spatial derivative of the second equation and positing a plane wave solution $h^\prime = \mr e^{\mr i(kx - \omega t)}$, we obtain
\begin{equation*}
	\omega = \pm\frac{\sqrt{gh_0}k}{1 + \frac{1}{6}(h_0k)^2}
\end{equation*}
as angular frequency.
Therefore, the phase velocity $c = \frac{\omega}{k}$ of the BBM-BBM equations normalized by $c_0 = \sqrt{gh_0}$ is given by
\begin{equation*}
	\frac{c_{\text{BBM-BBM}}}{\sqrt{gh_0}} = \frac{1}{1 + \frac{1}{6}(h_0k)^2}.
\end{equation*}
The normalized wave speed of the BBM-BBM equations and the one for the more general Euler equations \cite[Chapter 12]{whitham1974linear}
\begin{equation}\label{eq:disp-rel-Euler}
	\frac{c_{\text{Euler}}}{\sqrt{gh_0}} = \sqrt{\frac{\tanh(h_0k)}{h_0k}}
\end{equation}
are plotted and compared in \Cref{fig:disp-rels} (including the dispersion relation of the model of Svärd and Kalisch discussed in the following section).

\subsection{Dispersive model of Svärd and Kalisch}
Recently, a new dispersive shallow water model was presented in
a preprint \cite{svard2023novel} and published in \cite{svard2025novel}.
The model described in \cite{svard2025novel} is a special case of the
previous model examined in \cite{svard2023novel}, which can be obtained
by choosing the model parameters $\tilde\alpha = \tilde\gamma = 0$ (see below).
Notably, the simplified model is isotropic, while the more general one is not.
However, as the general model is mathematically a strict generalization of the
simpler model, it also satisfies an entropy condition, and has a more flexible
dispersive behavior, we will consider that one in this article.


For constant $b$, Svärd and Kalisch \cite{svard2023novel} propose the model
\begin{subequations}
	\begin{align}
		h_t + (hv)_x &= \alpha(h + b)_{xxx},\label{eq:svaerd-kalisch-constant1}\\
		(hv)_t + (hv^2)_x + gh(h + b)_x &= \alpha(v(h + b)_{xx})_x + \beta v_{xxt} + \gamma v_{xxx}\label{eq:svaerd-kalisch-constant2},
	\end{align}
\end{subequations}
adding third-order derivative terms to the classical SWEs \eqref{eq:swe-mass}--\eqref{eq:swe-momentum} weighted by coefficients $\alpha, \beta, \gamma$. Dedimensionalizing the equations gives the dimensionless coefficients $\tilde\alpha, \tilde\beta, \tilde\gamma$ satisfying
\begin{equation}\label{eq:SK-coefficients}
	\alpha = \tilde\alpha\sqrt{gh_0}h_0^2,\quad\beta = \tilde\beta h_0^3,\quad\gamma = \tilde\gamma\sqrt{gh_0}h_0^3,
\end{equation}
for a reference height $h_0$. The dispersion relation yields a quadratic equation in the frequency $\omega$ as described in \cite{svard2023novel}.
Svärd and Kalisch propose to optimize $\tilde\alpha, \tilde\beta, \tilde\gamma$, such that the resulting dispersion relation $\omega(k)$ follows the one of the Euler equations. The optimization, as performed in \cite{svard2023novel}, can be done by matching the first few terms of the Taylor series for the dispersion relation of the Euler equations resulting in \emph{set~1} \cite{svard2023novel}
\begin{equation*}
	\tilde\alpha = -\frac{1}{3}, \quad \tilde\beta = \tilde\gamma = 0.
\end{equation*}
Another possibility is to minimize some norm of the relative error
\begin{equation*}
	\varepsilon(k) = \frac{|\omega(k) - \omega_{\text{Euler}}(k)|}{\omega_{\text{Euler}}(k)}
\end{equation*}
for some range of $k$. Taking $k\in(0, 2\pi)$ gives \emph{set~2} \cite{svard2023novel}
\begin{equation*}
\begin{gathered}
	\tilde\alpha = 0.0004040404040404049,
	\quad
	\tilde\beta = 0.49292929292929294,
	\\
	\tilde\gamma = 0.15707070707070708
\end{gathered}
\end{equation*}
and taking $k\in(0, 8\pi)$ results in the following \emph{set~3} \cite{svard2023novel}
\begin{equation*}
	\tilde\alpha = 0, \quad\tilde\beta = 0.27946992481203003, \quad\tilde\gamma = 0.0521077694235589.
\end{equation*}
Instead of minimizing the relative error, Svärd and Kalisch also minimized the absolute error $|\omega(k) - \omega_{\text{Euler}}(k)|$, which leads for $k\in(0, 2\pi)$ to \emph{set~4} \cite{svard2023novel}
\begin{equation*}
	\tilde\alpha = 0, \quad\tilde\beta = 0.2308939393939394, \quad\tilde\gamma = 0.04034343434343434.
\end{equation*}
The simplified model in \cite{svard2025novel} corresponds to \emph{set~5}
\begin{equation*}
	\tilde\alpha = 0, \quad\tilde\beta = 1/3, \quad\tilde\gamma = 0.
\end{equation*}
The resulting dispersion relations are plotted in \Cref{fig:disp-rels}.

\begin{figure}[htbp]
	\centering
	\includegraphics[width=0.9\textwidth]{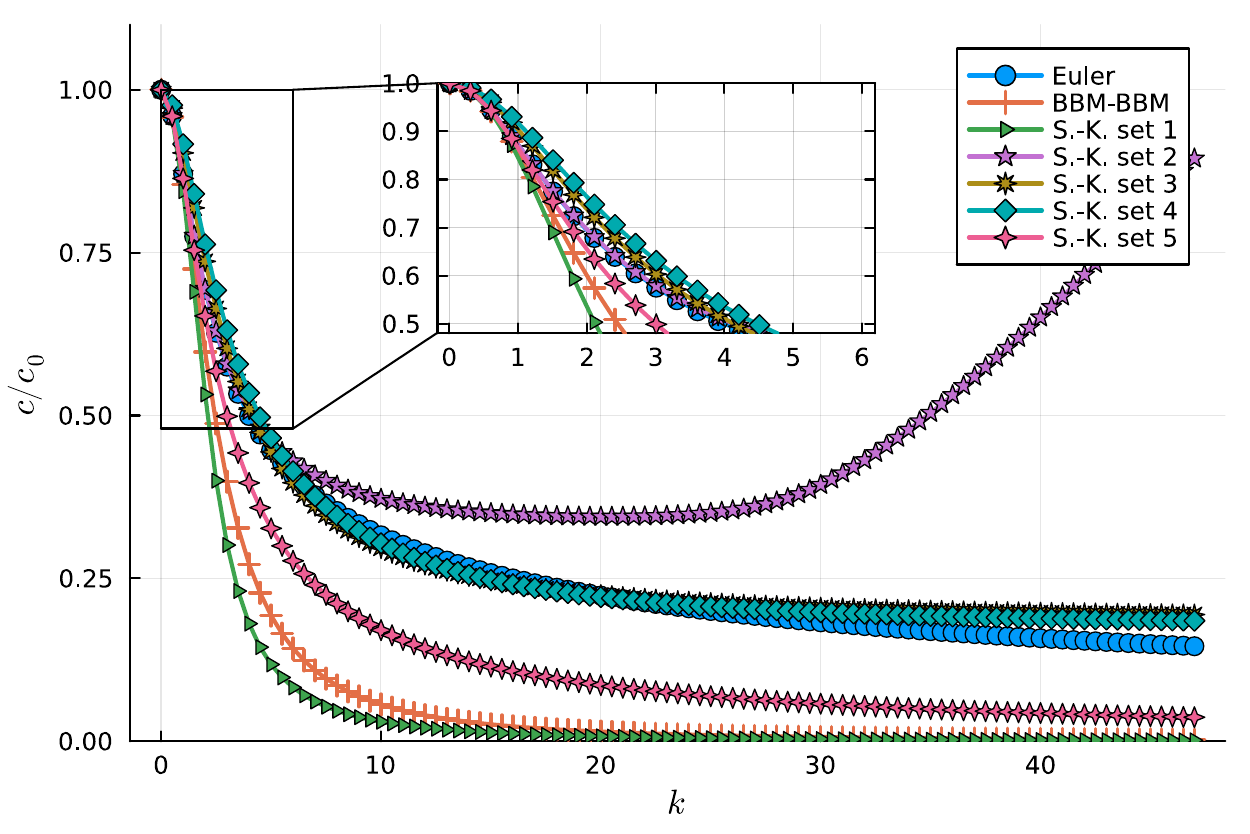}
	\caption{Dispersion relations of different models for $k\in(0, 15\pi)$}
	\label{fig:disp-rels}
\end{figure}

All of the models approximate the dispersion relation of the Euler equations \eqref{eq:disp-rel-Euler} reasonably well for high wavelengths (thus small wavenumbers) as can be seen in the zoomed-in plot. For higher wavenumbers the optimized coefficients for the Svärd-Kalisch model, i.e., sets 2, 3, and 4, are closer to the optimal dispersion relation. While the wavespeed diverges for set~2 if $k\to\infty$, sets 3 and 4 closely follow the dispersion relation of the Euler equations over a long range of wavenumbers and seem to converge to a value $> 0$.
An advantage of the model of Svärd and Kalisch lies in its flexibility and in the fact that the dispersion relation can be adjusted by tuning the parameters $\tilde\alpha, \tilde\beta, \tilde\gamma$ for the relevant range of wavelengths.

For the generalization of \eqref{eq:svaerd-kalisch-constant1}--\eqref{eq:svaerd-kalisch-constant2} to a general bottom topography,
Svärd and Kalisch demand that the dispersive model approaches the SWEs with smaller water height since in that case dispersion becomes less relevant. Hence, one possibility is to exchange the constant coefficients $\alpha$, $\beta$, $\gamma$ by space-dependent coefficients $\hat\alpha(x)$, $\hat\beta(x)$ and $\hat\gamma(x)$ that are small for small water depths.
A straight-forward modification of \eqref{eq:svaerd-kalisch-constant1}--\eqref{eq:svaerd-kalisch-constant2} then reads as
\begin{subequations}
	\begin{align}
		h_t + (hv)_x &= (\hat\alpha(\hat\alpha(h + b)_x)_x)_x,\label{eq:svaerd-kalisch1}\\
		(hv)_t + (hv^2)_x + gh(h + b)_x &= (\hat\alpha v(\hat\alpha(h + b)_x)_x)_x + (\hat\beta v_x)_{xt} + \frac{(\hat\gamma v_x)_{xx} + (\hat\gamma v_{xx})_x}{2}.\label{eq:svaerd-kalisch2}
	\end{align}
\end{subequations}
As a natural definition, the nondimensional form \eqref{eq:SK-coefficients} is generalized to
\begin{equation}\label{eq:SK-coefficients2}
	\hat\alpha^2(x) = \tilde\alpha\sqrt{gD}D^2, \quad\hat\beta(x) = \tilde\beta D^3, \quad \hat\gamma(x) = \tilde\gamma\sqrt{gD}D^3.
\end{equation}
for the still water depth $D$. With these definitions we have $\hat\alpha$, $\hat\beta$,
$\hat\gamma\to 0$ if $D\to 0$, i.e., for small water height the equations turn into the
SWEs. This property makes it unnecessary to switch between the dispersive model and the
SWEs near the shore, as is, e.g., done in \cite{berger2021towards}.
By the definitions \eqref{eq:SK-coefficients2} we get the additional constraint
$\tilde\alpha\ge 0$. This condition does not hold for set~1, but set~1 was outperformed by
the sets 2, 3, and 4 anyway.

The entropy analysis of \eqref{eq:svaerd-kalisch1}--\eqref{eq:svaerd-kalisch2}
presented in \cite{svard2023novel} shows that $U$ \eqref{eq:entropy-swe}
is a partial entropy with corresponding partial entropy flux $F$ \eqref{eq:entropy-flux-swe}.
In contrast to the SWEs the source term $-mz_x$ in \eqref{eq:entropy-equality-balance-law} does not vanish, but consists of contributions $M_\alpha$, $M_\beta$, $M_\gamma$ from the dispersive terms.
\begin{theorem}[Svärd and Kalisch (2023)]\label{th:entropy-cons-SK}
	Any smooth solution $\bm u$ of the equations \eqref{eq:svaerd-kalisch1}--\eqref{eq:svaerd-kalisch2} with periodic boundary conditions satisfies
	\begin{equation}\label{eq:entropy-balance-law-SK}
		U(\bm u)_t + F(\bm u)_x = M_\alpha + M_\beta + M_\gamma,
	\end{equation}
	where the entropy pair $(U, F)$ is the same as for the SWEs \eqref{eq:entropy-swe}--\eqref{eq:entropy-flux-swe}, and the dispersive contributions are given by
	\begin{align*}
		M_\alpha &= g((h + b)\hat\alpha(\hat\alpha(h + b)_x)_x)_x - \frac{1}{2}g(\hat\alpha (h + b)^2_x)_x + \frac{1}{2}(v^2\hat\alpha(\hat\alpha(h + b)_x)_x)_x,\\
		M_\beta &= (\hat\beta vv_{xt})_x - \frac{1}{2}(\hat\beta v_x^2)_t,\\
		M_\gamma &= \frac{1}{2}\left((v(\hat\gamma v_x)_x)_x + (v\hat\gamma v_{xx})_x - (\hat\gamma v_x^2)_x\right).
	\end{align*}
\end{theorem}
\begin{proof}
    See \cite[pp. 12--13]{svard2023novel}.
\end{proof}
The relation \eqref{eq:entropy-balance-law-SK} can be interpreted in a different way. Since all of the terms in $M_\alpha$, $M_\beta$ and $M_\gamma$ contain either an outer spatial derivative or an outer time derivative, we can see that the model of Svärd and Kalisch admits the mathematical entropy function
\begin{equation}\label{eq:mod-entropy}
	\hat U(\bm u) = U(\bm u) + \frac{1}{2}\hat\beta v_x^2,
\end{equation}
which we call \emph{modified entropy}. The corresponding \emph{modified entropy flux}
\begin{equation*}
	\hat F(\bm u) = F(\bm u) - F_\alpha(\bm u) - F_\beta(\bm u) - F_\gamma(\bm u)
\end{equation*}
consists of the entropy flux $F$ for the SWEs and additional dispersive contributions to the modified entropy flux that are given by
\begin{align*}
	F_\alpha(\bm u) &= g(h + b)(\hat\alpha(\hat\alpha(h + b))_x)_x - \frac{1}{2}g\hat\alpha (h + b)^2_x + \frac{1}{2}v^2\hat\alpha(\hat\alpha(h + b)_x)_x,\\
	F_\beta(\bm u) &= \hat\beta vv_{xt},\\
	F_\gamma(\bm u) &= \frac{1}{2}\left(v(\hat\gamma v_x)_x + v\hat\gamma v_{xx} - \hat\gamma v_x^2\right).
\end{align*}
We note that the modified entropy and entropy flux include derivatives and therefore take the form of operators. Due to \Cref{th:entropy-cons-SK}, the entropy conservation of the modified entropy can be written as
\begin{equation*}
	\hat U(\bm u)_t + \hat F(\bm u)_x = 0.
\end{equation*}
As an immediate consequence, the \emph{total modified entropy}
\begin{equation}\label{eq:total-mod-entropy}
	\hat\cE(t; \bm u) = \int_\Omega\hat U(\bm u)\tr dx
\end{equation}
is conserved for smooth solutions $\bm u$ and periodic of reflecting (i.e., $v = 0$ on the boundary)
boundary conditions.
This entropy bound can be seen in analogy to an entropy inequality and can be used as a stability criterion, especially for numerical methods.

Although a linear stability analysis has been performed in \cite{svard2023novel}, there is no general well-posedness result for the equations \eqref{eq:svaerd-kalisch1}--\eqref{eq:svaerd-kalisch2}.

\section{Structure-preserving discretizations}\label{sect:numerical-methods}

In this section, we first review the concept of SBP operators.
Afterwards, we develop new structure-preserving methods for the BBM-BBM
equations and the model of Svärd and Kalisch.

\subsection{Review of summation-by-parts operators}
In recent years, SBP operators have gained particular interest as they allow to transfer analytical results to numerical methods. This is done by mimicking integration by parts discretely, which is one of the key ingredients at the continuous level.

SBP operators were first developed for finite difference methods
\cite{kreiss1974finite,strand1994summation,mattsson2004summation}
to mimic stability proofs based on integration by parts as used in
finite element methods. However, exact integration can be impossible or
expensive in finite element methods. In this case, SBP formulations can be
advantageous since they include a quadrature rule.
In particular, split forms \cite{fisher2013discretely} can be used with
SBP operators to avoid the need for exact integration and the lack of
discrete chain and product rules \cite{ranocha2019mimetic}.
Several classes of numerical methods can be formulated via SBP operators,
including finite volume methods \cite{nordstrom2001finite,nordstrom2003finite},
continuous Galerkin methods \cite{hicken2016multidimensional,hicken2020entropy,abgrall2020analysisI},
discontinuous Galerkin (DG) methods \cite{gassner2013skew,carpenter2014entropy,chan2018discretely},
and flux reconstruction methods \cite{huynh2007flux,ranocha2016summation}.
Further references are given in the review articles
\cite{svard2014review,fernandez2014review,chen2020review}.

In this work, we focus on the one-dimensional case. However, extensions to two or three spatial dimensions can be obtained using either
tensor product constructions or multidimensional SBP operators
\cite{fernandez2014review,hicken2016multidimensional,fernandez2018simultaneous,hicken2025constructing}.

The definitions and results presented in this chapter are mainly based on the formulations of \cite{ranocha2021broad}.
Multiplication of discrete functions $\bm u$ and $\bm v$ is performed component-wise $(\bm u\bm v)_j = u_jv_j$ and similarly $(\bm u^k)_j = u_j^k$. We also define $\bm 1 = (1, \ldots, 1)^T$, $\bm 0 = (0, \ldots, 0)^T$, $\bm e_L = (1, 0, \ldots, 0)^T$, and $\bm e_R = (0, \ldots, 0, 1)^T$.

SBP operators describe $i$-th derivatives $\frac{\tr d^i}{\tr dx^i}$, $i\in\mathbb{N}$, by a \emph{derivative matrix} $D_i$ and integration by a \emph{mass matrix} (or \emph{norm matrix}) $M$.
\begin{definition}
    Given a grid $\bm x$, a \emph{$p$-th order accurate $i$-th derivative matrix}, $D_i$ \cite[Definition 2.1]{ranocha2021broad} is a matrix that satisfies for all $k = 0, \ldots, p$
    \begin{equation}\label{eq:order}
        D_i\bm x^k = k(k - 1)\ldots(k - i + 1)\bm x^{k - i}.
    \end{equation}
    The derivative matrix $D_i$ is called \emph{consistent} if $p\ge 0$,
	i.e., $D_i\bm 1 = \bm 0$.
\end{definition}
Note that in the case $k < i$ in \eqref{eq:order} one of the previous factors is zero and we interpret the right-hand side as zero to avoid any potential issues with negative exponents.
\begin{definition}\label{def:first-derivative-SBP}
    A \emph{first-derivative summation-by-parts (SBP) operator} \cite[Definition 2.2]{ranocha2021broad} consists of a grid $\bm x$, a consistent first-derivative matrix $D_1$ and a symmetric and positive-definite matrix $M$, such that
    \begin{equation}\label{eq:1-sbp}
        MD_1 + D_1^TM = \bm e_R\bm e_R^T - \bm e_L\bm e_L^T.
    \end{equation}
    A \emph{periodic first-derivative SBP operator} \cite[Definition 2.3]{ranocha2021broad} consists of a grid $\bm x$, a consistent first-derivative matrix $D_1$ and a symmetric and positive-definite matrix $M$, such that
    \begin{equation}\label{eq:1-per-sbp}
        MD_1 + D_1^TM = 0.
    \end{equation}
\end{definition}
Sometimes we also refer to the derivative matrix $D_i$ as an SBP operator.
Introducing the notation $\langle\bm u, \bm v\rangle_M = \bm u^TM\bm v$ for the inner product that is induced by $M$, we can see that a periodic first-derivative SBP operator satisfies
\begin{equation*}
    \langle D_1\bm u, \bm v\rangle_M = -\langle\bm u, D_1\bm v\rangle_M.
\end{equation*}
The symmetry of $M$ yields $\langle\bm u, \bm v\rangle_M = \langle\bm v, \bm u\rangle_M$ and from the positive-definiteness it follows $\langle\bm u, \bm u\rangle_M \ge 0$ and $\langle\bm u, \bm u\rangle_M = 0 \Leftrightarrow\bm u = \bm 0$.

\begin{example}\label{ex:first-sbp-operators}
    Consider a uniform grid $\bm x$ with $\Delta x = x_{i + 1} - x_i = \frac{x_{\text{max}} - x_{\text{min}}}{N - 1}$, $x_i = x_{\text{min}} + (i - 1)\Delta x$. Then the classical central finite difference operator with special treatment of the boundaries
    \begin{equation*}
        D_1 = \frac{1}{\Delta x}\begin{pmatrix}
            -1 & 1 & & & \\
            -\frac{1}{2} & 0 & \frac{1}{2} & & \\
            & \ddots & \ddots & \ddots & \\
            & & -\frac{1}{2} & 0 & \frac{1}{2}\\
            & & & -1 & 1
        \end{pmatrix},\quad
        M = \Delta x\begin{pmatrix}
            \frac{1}{2} & & & & \\
            & 1 & & & \\
            & & \ddots & & \\
            & & & 1 & \\
            & & & & \frac{1}{2}
        \end{pmatrix}
    \end{equation*}
    is a second-order accurate first-derivative SBP operator. Similarly, a second-order accurate periodic first-derivative SBP operator is given by $M = \Delta x I$ and
    \begin{equation*}
        D_1 = \frac{1}{2\Delta x}\begin{pmatrix}
            0 & 1 & 0 & \ldots & 0 & -1\\
            -1 & 0 & 1 & \ddots & 0 & 0\\
            & \ddots & \ddots & \ddots & & \\
            & & \ddots & \ddots & \ddots & & \\
            0 & 0 & & -1 & 0 & 1\\
            1 & 0 & \ldots & 0 & -1 & 0
        \end{pmatrix}.
    \end{equation*}
\end{example}

In \cite{mattsson2017diagonal} so-called upwind operators are introduced that are commonly used to allow entropy-dissipative numerical methods as we will see later.
\begin{definition}
    A \emph{first-derivative upwind SBP operator} \cite[Definition 2.4]{ranocha2021broad} consists of a grid $\bm x$, consistent first-derivative matrices $D_{1,\pm}$ and a symmetric and positive-definite matrix $M$, such that
    \begin{equation*}
        MD_{1,+} + D_{1,-}^TM = \bm e_R^T\bm e_R - \bm e_L^T\bm e_L
    \end{equation*}
    and $\frac{1}{2}M(D_{1,+} - D_{1,-})$ being negative semidefinite.\\
    A \emph{periodic first-derivative upwind SBP operator} \cite[Definition 2.5]{ranocha2021broad} consists of a grid $\bm x$, consistent first-derivative matrices $D_{1,\pm}$ and a symmetric and positive-definite matrix $M$, such that
    \begin{equation*}
        MD_{1,+} + D_{1,-}^TM = 0
    \end{equation*}
    and $\frac{1}{2}M(D_{1,+} - D_{1,-})$ being negative semidefinite.
\end{definition}
Closely related to upwind operators are so-called \emph{dual-pair} (DP) SBP operators, which were first introduced in \cite{dovgilovich2015high}.
In terms of inner products, periodic upwind operators satisfy
\begin{equation*}
    \langle D_{1,+}\bm u, \bm v\rangle_M = -\langle\bm u, D_{1,-}\bm v\rangle_M.
\end{equation*}
A (central) first-derivative SBP operator can be constructed from first-derivative upwind SBP operators by defining $D_1 = \frac{1}{2}(D_{1,+} + D_{1,-})$.
\begin{definition}\label{def:second-derivative-SBP}
    A \emph{second-derivative SBP operator} \cite[Definition 2.6]{ranocha2021broad} consists of a grid $\bm x$, a consistent second-derivative matrix $D_2$, a symmetric and positive-definite matrix $M$, and derivative vectors $\bm d_{L/R}$ approximating the first derivative at the left/right endpoint as $\bm d_{L/R}^T\bm u \approx u^\prime(x_{\text{min}/\text{max}})$, such that
    \begin{equation*}
        MD_2 = D_2^TM + \bm e_R\bm d_R^T - \bm e_L\bm d_L^T - \bm d_R\bm e_R^T + \bm d_L\bm e_L^T.
    \end{equation*}
    A \emph{periodic second-derivative SBP operator} \cite[Definition 2.7]{ranocha2021broad} consists of a grid $\bm x$, a consistent second-derivative matrix $D_2$, and a symmetric and positive-definite matrix $M$, such that
    \begin{equation}\label{eq:2-per-sbp}
        MD_2 = D_2^TM.
    \end{equation}
\end{definition}
For periodic second-derivative SBP operators we have $\langle D_2\bm u, \bm v\rangle_M = \langle\bm u, D_2\bm v\rangle_M$.
SBP operators of higher derivative order can be defined similarly, see e.g. \cite{mattsson2014diagonal}.

\begin{remark}\label{rem:SBP_grids}
    SBP operators can be interpreted as classical finite difference operators, where the mass matrix of the periodic
    operators on a regular grid is given by $M = \Delta xI$. SBP operators with boundary-adapted grids are developed
    in \cite{mattsson2018boundary}. Spectral finite difference methods \cite{carpenter1996spectral} can be obtained by
    using a Lobatto-Legendre basis for polynomials of degree $p$. Here, the grid is non-uniform and given by the
    Lobatto-Legendre quadrature nodes and the mass matrix is given as the diagonal matrix $M = \diag(\omega_1,\ldots,\omega_p)$,
    where $\omega_i$ are the Lobatto-Legendre quadrature weights. Similarly, the discontinuous Galerkin spectral element
    method (DGSEM) \cite{gassner2013skew} is formulated by coupling local SBP operators based on Lobatto-Legendre
    quadrature \cite{ranocha2021broad}.
\end{remark}

In the following, we will use SBP operators with diagonal norm matrix $M$.
In this case, we have $\langle\bm u\bm v,\bm w\rangle_M = \langle\bm u,\bm v\bm w\rangle_M$.
Concretely, we will use finite difference, spectral difference
\cite{carpenter1996spectral}, discontinuous Galerkin spectral element
methods (DGSEM) \cite{gassner2013skew}, and the corresponding continuous
Galerkin methods as described in \cite{ranocha2021broad} and references
cited therein.

We collect some properties of SBP operators in the following lemma.
\begin{lemma}\label{le:per-SBP-prop}
    Periodic first- and second-derivative SBP operators $D_i$, $i = 1, 2$ and periodic upwind SBP operators $D_{1,+}$, $D_{1,-}$ satisfy
    \begin{enumerate}
        \item $\bm 1^TMD_i = \bm 0^T$ and $\bm 1^TMD_{1, \pm} = \bm 0^T$, \cite[Lemma 2.1]{ranocha2021broad}
        \item $\langle D_1\bm u, \bm u\rangle_M = 0$,
        \item $\langle D_{1,-}\bm u, \bm u\rangle_M\ge 0$ and $\langle D_{1,+}\bm u,\bm u\rangle_M\le 0$.
    \end{enumerate}
\end{lemma}
\begin{proof}
	See \cite[Lemma 2.1]{ranocha2021broad} for 1. Items 2. and 3. follow from the symmetry of $M$ and the negative-definiteness of $\frac{1}{2}M(D_{1,+} - D_{1,-})$.
\end{proof}

\subsection{Semidiscretization of the BBM-BBM equations with variable topography}\label{sect:num-BBM-BBM}
An energy-conservative discretization of the BBM-BBM equations with constant bathymetry has been studied in \cite{ranocha2021broad,ranocha2021rate}. The semidiscretization for periodic boundary conditions can be written as
\begin{subequations}
\begin{align}
    \bm\eta_t &= -\left(I - \frac{1}{6}D^2D_2\right)^{-1}D_1(D\bm v + \bm\eta\bm v),\label{eq:semi-BBM-BBM-const1}\\
    \bm v_t &= -\left(I - \frac{1}{6}D^2D_2\right)^{-1}D_1\left(g\bm\eta + \frac{1}{2}\bm v^2\right).\label{eq:semi-BBM-BBM-const2}
\end{align}
\end{subequations}
Further discretizations of the BBM-BBM equations with constant bottom topography are available, e.g., \cite{antonopoulos2010numerical,mitsotakis2021conservative}.
In this contribution, we generalize \eqref{eq:semi-BBM-BBM-const1}--\eqref{eq:semi-BBM-BBM-const2} to the case of a variable topography, while ensuring energy conservation.

In \cite{israwi2021regularized}, the authors use a finite element discretization of \eqref{eq:BBMBBM1}--\eqref{eq:BBMBBM2} that provably converges to the exact solution, but the proposed method does not conserve the energy exactly. In the following, we put forward an energy-conserving method based on SBP operators. To do so, we define the diagonal matrix $K = \diag(\bm D^2)$ for $\bm D = (D(x_1), \ldots, D(x_N))^T$. Then, the semidiscretization is formulated as
\begin{subequations}
\begin{align}
    \bm\eta_t + \left(I - \frac{1}{6}D_1 KD_1\right)^{-1}D_1(\bm D\bm v + \bm\eta\bm v) &= \bm 0,\label{eq:semi-BBM-BBM1}\\
    \bm v_t + \left(I - \frac{1}{6}D_2K\right)^{-1}D_1\left(g\bm\eta + \frac{1}{2}\bm v^2\right) &= \bm 0,\label{eq:semi-BBM-BBM2}
\end{align}
\end{subequations}
where $D_1$ and $D_2$ are periodic first-derivative and second-derivative SBP operators, respectively.
Particularly, for the case of $D(x) \equiv D_\text{const}$ being constant (resulting in $K = D_\text{const}^2I$), the semidiscretization
coincides with \eqref{eq:semi-BBM-BBM-const1}--\eqref{eq:semi-BBM-BBM-const2} taking $D_2 = D_1^2$
in the first equation. Due to the more complex structure of the elliptic operators compared to the
case of a constant bathymetry, an analogous version of the proof in \cite{ranocha2021broad}
(specifically Lemma 2.2 and Lemma 2.3) does not hold anymore. Thus, we turn back to mimicking the
continuous form of the energy conservation proof on a discrete level. This requires to take
$D_2 = D_1^2$. Under this condition, the semidiscretization \eqref{eq:semi-BBM-BBM1}, can be rewritten as
\begin{subequations}
\begin{align}
    \bm\eta_t + D_1\underbrace{\left(\bm D\bm v + \bm\eta\bm v - \frac{1}{6}KD_1\bm\eta_t\right)}_{=:\bm{F_1}} &= \bm 0,\label{eq:semi-BBM-BBM-cons1}\\
    \bm v_t + D_1\underbrace{\left(g\bm\eta + \frac{1}{2}\bm v^2 - \frac{1}{6}D_1K\bm v_t\right)}_{=:\bm{F_2}} &= \bm 0.\label{eq:semi-BBM-BBM-cons2}
\end{align}
\end{subequations}
This form mimics the conservation form of the BBM-BBM equations and can be utilized to prove the energy conservation of the numerical scheme. In the context of SBP operators, the discrete total energy is given by
\begin{equation}\label{eq:discrete-total-entropy}
	\cE(t; \bm\eta, \bm v) = \bm 1^TM\bm U,
	\qquad
	\bm U(\bm\eta, \bm v) = \frac{1}{2}g\bm\eta^2 + \frac{1}{2}(\bm\eta + \bm D)\bm v^2,
\end{equation}
where $\bm U$ is the discrete energy, cf.\ \eqref{eq:entropy-swe-prim}.
The discrete total mass and velocity are
\begin{equation}
\label{eq:discrete-total-mass-and-velocity}
	\cM(t; \bm\eta) = \bm 1^TM\bm\eta,
	\qquad
	\cV(t; \bm v) = \bm 1^TM\bm v.
\end{equation}
\begin{theorem}\label{th:semi-props-BBMBBM}
    If $D_1$ is a periodic first-derivative SBP operator and $D_2$ is a periodic second-derivative SBP operator, then the semidiscretization \eqref{eq:semi-BBM-BBM1}--\eqref{eq:semi-BBM-BBM2} conserves the linear invariants $\cM$ and $\cV$ \eqref{eq:discrete-total-mass-and-velocity}. If $D_2 = D_1^2$, the quadratic invariant $\cE$ \eqref{eq:discrete-total-entropy} is also conserved. Furthermore, the scheme is well-balanced, i.e., exactly preserves the lake-at-rest solution.
\end{theorem}
\begin{proof}
    The conservation of $\cM$ and $\cV$ follows immediately by \Cref{le:per-SBP-prop} item 1:
    \begin{equation*}
        \frac{\tr d}{\tr dt}\cM(t; \bm\eta) = \bm 1^TM\bm\eta_t = -\bm 1^TMD_1\bm F_1 = 0
    \end{equation*}
    and similarly for $\cV$.\\
    Because $K$ and $M$ are diagonal, they commute and therefore we have:
    \begin{equation*}
        \langle\bm\eta_t, D_1K\bm v_t\rangle_M = -\langle KD_1\bm\eta_t, \bm v_t\rangle_M.
    \end{equation*}
    The assumption $D_2 = D_1^2$ allows us to use this together with \eqref{eq:semi-BBM-BBM-cons1}--\eqref{eq:semi-BBM-BBM-cons2} to obtain
    \begin{align*}
        &\quad
		\frac{\tr d}{\tr dt}\cE(t; \bm\eta, \bm v)
		= \frac{1}{2}\frac{\tr d}{\tr dt}\left(\langle g\bm\eta, \bm\eta\rangle_M + \langle\bm v^2, \bm D + \bm\eta\rangle_M\right)\\
        &= \langle g\bm\eta_t, \bm\eta\rangle_M + \frac{1}{2}\langle\bm v^2, \bm\eta_t\rangle_M + \langle \bm v\bm v_t, \bm D + \bm\eta\rangle_M - \frac{\langle\bm\eta_t, D_1K\bm v_t\rangle_M - \langle KD_1\bm\eta_t, \bm v_t\rangle_M}{6}\\
        &= \langle\bm\eta_t, g\bm\eta + \frac{1}{2}\bm v^2 - \frac{1}{6}D_1K\bm v_t\rangle_M + \langle\bm v_t, (\bm D + \bm\eta)\bm v - \frac{1}{6}KD_1\bm\eta_t\rangle_M \\
        &= \langle\bm\eta_t, \bm {F_2}\rangle_M + \langle\bm v_t, \bm {F_1}\rangle_M\\
        &= -\langle D_1\bm{F_1}, \bm{F_2}\rangle_M - \langle D_1\bm{F_2}, \bm{F_1}\rangle_M = \langle \bm{F_1}, D_1\bm{F_2}\rangle_M - \langle \bm{F_1}, D_1\bm{F_2}\rangle_M
        = 0.
    \end{align*}
    Well-balancedness is an immediate consequence of the consistency of $D_1$.
\end{proof}
An alternative semidiscretization to \eqref{eq:semi-BBM-BBM1}--\eqref{eq:semi-BBM-BBM2} that uses upwind operators instead of central first-derivative operators is given by
\begin{subequations}
\begin{align}
    \bm\eta_t + \left(I - \frac{1}{6}D_{1,-} KD_{1,+}\right)^{-1}D_{1,-}(\bm D\bm v + \bm\eta\bm v) &= \bm 0,\label{eq:semi-up-BBM-BBM1}\\
    \bm v_t + \left(I - \frac{1}{6}D_{1,+}D_{1,-}K\right)^{-1}D_{1,+}\left(g\bm\eta + \frac{1}{2}\bm v^2\right) &= \bm 0.\label{eq:semi-up-BBM-BBM2}
\end{align}
\end{subequations}
\begin{theorem}
	If $D_{1,\pm}$ form a periodic first-derivative upwind SBP operator, then the semidiscretization \eqref{eq:semi-up-BBM-BBM1}--\eqref{eq:semi-up-BBM-BBM2} conserves the linear invariants $\cM$ and $\cV$ \eqref{eq:discrete-total-mass-and-velocity} and the quadratic invariant $\cE$ \eqref{eq:discrete-total-entropy}. Furthermore, the scheme is well-balanced.
\end{theorem}
\begin{proof}
	The proof can be obtained in the same way as in \Cref{th:semi-props-BBMBBM}.
\end{proof}
The scheme \eqref{eq:semi-up-BBM-BBM1}--\eqref{eq:semi-up-BBM-BBM2} has the advantage that the wide-stencil operators $D_1KD_1$ and $D_1^2K$ are replaced by consecutively applying one upwind operator in positive direction and one in negative direction, which often results in a better approximation, cf. \Cref{fig:errors-stencils}.
Note that swapping all $D_{1,-}$ and $D_{1,+}$ also yields an energy-conservative scheme.
In this semidiscretization, the discrete derivative operators applied to
the water height and the velocity are different. This is also the case in
the semidiscretization of the classical SWEs of
\cite{hew2023strongly}.

Next, we consider reflecting boundary conditions, i.e., \eqref{eq:BBMBBM1}--\eqref{eq:BBMBBM2} for $t > 0$,
$x\in (x_\text{min}, x_\text{max})$ and
\begin{equation*}
	\eta_x(t, x) = 0, \qquad
	v(t, x) = 0,
\end{equation*}
for $t > 0$ and $x\in\{x_\text{min}, x_\text{max}\}$. Similar to \cite{ranocha2021broad} we construct SBP operators
that impose the homogeneous Dirichlet boundary condition for $v$ strongly and the homogeneous Neumann boundary
condition of $\eta$ in a weak-strong sense. To this end, we define $P_D = \diag(0, 1, \ldots, 1, 0)$ as the projection
to the inner nodes. Then we have $P_D\bm e_{R/L}\bm e_{R/L}^T = 0$. Define for some $\delta\in\R$ and a first-derivative
SBP operator $D_1$ the operators with
\begin{itemize}
	\item homogeneous Dirichlet boundary conditions $I - \delta D_{2,D}K$, which satisfies for $\bm v, \bm w\in \R^N$:
	\begin{equation*}
		\bm v = (I - \delta D_{2,D}K)^{-1}\bm w \Leftrightarrow P_D(I - \delta D_{2,D}K)\bm v = P_D\bm w\quad\text{and}\quad (I - P_D)\bm v = \bm 0,
	\end{equation*}
	i.e., $v_1 = v_N = 0$ and if $I - \delta D_{2,D}$ induced by $I - \delta D_1^2$ also $P_DD_{2,D} = P_DD_1^2$ and
	\item weak-strong imposition of homogeneous Neumann boundary conditions $I - \delta D_{2,N}^K$, which satisfies for $\bm v, \bm w\in \R^N$:
	\begin{equation*}
		\bm v = (I - \delta D_{2,N}^K)^{-1}\bm w \Leftrightarrow (I + \delta M^{-1}D_1^TMP_DKD_1)\bm v = \bm w.
	\end{equation*}
	Since $D_1^TMP_D = -MD_1P_D + \bm e_R\bm e_R^TP_D - \bm e_L\bm e_L^TP_D = -MD_1P_D$, we have $D_{2,N}^K = D_1P_DKD_1$.
\end{itemize}
We note that the idea of the operator $D_{2,D}$ is very similar to the Dirichlet-SBP operator analyzed in \cite{gjesteland2022entropy}.
Then the semidiscretization reads
\begin{subequations}
	\begin{align}
		\bm\eta_t &= -\left(I - \frac{1}{6}D_{2,N}^K\right)^{-1}D_1(\bm D\bm v + \bm\eta\bm v),\label{eq:semi-BBM-BBM-reflecting1}\\
		\bm v_t &= -\left(I - \frac{1}{6}D_{2,D}K\right)^{-1}D_1\left(g\bm\eta + \frac{1}{2}\bm v^2\right)\label{eq:semi-BBM-BBM-reflecting2}.
	\end{align}
\end{subequations}
Assuming that $P_DD_{2,D} = P_DD_1^2$ we can write \eqref{eq:semi-BBM-BBM-reflecting1}--\eqref{eq:semi-BBM-BBM-reflecting2}
in first-order form similar to \eqref{eq:semi-BBM-BBM-cons1}--\eqref{eq:semi-BBM-BBM-cons2}
\begin{subequations}
	\begin{align}
		\bm\eta_t + D_1P_D\bm{F_1} &= \bm 0,\label{eq:semi-BBM-BBM-reflecting-first-order1}\\
		\bm v_t + P_DD_1\bm{F_2} &= \bm 0,\label{eq:semi-BBM-BBM-reflecting-first-order2}
	\end{align}
\end{subequations}
by additionally imposing $v_1 = v_N = 0$. Here, $\bm{F_1}$ and $\bm{F_2}$ are defined as before. The equivalence of
these two formulations follows by multiplying the second equation \eqref{eq:semi-BBM-BBM-reflecting2} with
$P_D$ and using $P_D\bm v_t = \bm v_t$ and $P_D(\bm D\bm v + \bm\eta\bm v) = (\bm D\bm v + \bm\eta\bm v)$ both
following from $P_D\bm v = \bm v$. With the first-order formulation we are positioned to prove the energy conservation
of the semidiscretization.
\begin{theorem}\label{th:semi-props-BBMBBM-reflecting}
	The semidiscretization given by \eqref{eq:semi-BBM-BBM-reflecting1}--\eqref{eq:semi-BBM-BBM-reflecting2} conserves
	the total mass and given that $P_DD_{2,D} = P_DD_1^2$, it also conserves the total energy. In addition, the scheme
	is well-balanced.
\end{theorem}
\begin{proof}
	If $v_1 = v_N = 0$ initially, we also have $\bm e_{R/L}\bm v_t = 0$ due to the Dirichlet boundary condition
	operator, which means that $\bm v$ satisfies the homogeneous boundary conditions for all times.
	The total mass is conserved due to
	\begin{equation*}
		\frac{\tr d}{\tr dt}\cM(t; \bm\eta) = \bm 1^TM\bm\eta_t = -\bm 1^TMD_1P_D\bm{F_1} = 0.
	\end{equation*}
	For the energy conservation, the same steps as in \Cref{th:semi-props-BBMBBM} result in
	\begin{align*}
		\frac{\tr d}{\tr dt}\cE(t; \bm\eta, \bm v) &= \langle\bm\eta_t, \bm {F_2}\rangle_M + \langle\bm v_t, \bm {F_1}\rangle_M
	\end{align*}
	and plugging in the semidiscretization \eqref{eq:semi-BBM-BBM-reflecting-first-order1}--\eqref{eq:semi-BBM-BBM-reflecting-first-order2} yields
	\begin{align*}
		\frac{\tr d}{\tr dt}\cE(t; \bm\eta, \bm v) &= -\langle D_1P_D\bm {F_1}, \bm {F_2}\rangle_M - \langle P_DD_1\bm{F_2}, \bm {F_1}\rangle_M
		\\
		&= \langle P_D\bm{F_1}, D_1\bm{F_2}\rangle_M - \langle \bm{F_1}, P_DD_1\bm{F_2}\rangle_M = 0,
	\end{align*}
	where in the second to last step the boundary terms vanish due to the multiplication by $P_D$ and the last
	step follows from the fact that $P_D$ is diagonal.

	Well-balancedness follows immediately from the consistency of $D_1$.
\end{proof}
We can formulate an upwind variant by using the operators $D_{2,D,\pm}$ and $D_{2,N,\mp}^K$ satisfying $D_{2,D,\pm} = P_DD_{1,+}D_{1,-}$ and $D_{2,N,\mp}^K = D_{1,-}P_DKD_{1,+}$ respectively:
\begin{subequations}
	\begin{align}
		\bm\eta_t &= -\left(I - \frac{1}{6}D_{2,N,\mp}^K\right)^{-1}D_{1,-}(\bm D\bm v + \bm\eta\bm v),\label{eq:semi-BBM-BBM-reflecting-upwind1}\\
		\bm v_t &= -\left(I - \frac{1}{6}D_{2,D,\pm}K\right)^{-1}D_{1,+}\left(g\bm\eta + \frac{1}{2}\bm v^2\right)\label{eq:semi-BBM-BBM-reflecting-upwind2}.
	\end{align}
\end{subequations}
This semidiscretization satisfies the first-order formulation (in addition to $v_1 = v_N = 0$)
\begin{align*}
	\bm\eta_t + D_{1,-}P_D\bm{F_{1,+}} &= \bm 0, &\quad
	\bm v_t + P_DD_{1,+}\bm{F_{2,-}} &= \bm 0
\end{align*}
where, similar to before, $\bm{F_{1,+}} = (\bm D + \bm\eta)\bm v - \frac{1}{6}KD_{1,+}\bm\eta_t$ and $\bm{F_{2,-}} = g\bm\eta + \frac{1}{2}\bm v^2 - \frac{1}{6}D_{1,-}K\bm v_t$. Again, a proof of the conservation of $\cM$ and $\cE$ as well as well-balancedness of the scheme can be conducted similarly to \Cref{th:semi-props-BBMBBM-reflecting}.

\subsection{Semidiscretization of the dispersive model of Svärd and Kalisch}\label{sect:num-svard-kalisch}
In \cite{svard2023novel} the authors present
a well-balanced and entropy-stable finite volume discretization of \eqref{eq:svaerd-kalisch1},
\eqref{eq:svaerd-kalisch2} with implicit-explicit time-stepping. Instead, we propose a new semidiscretization
based on SBP operators, which preserves the discrete total modified entropy
\begin{equation}\label{eq:discrete-total-mod-entropy}
    \hat\cE(t; \bm h, \bm P) = \bm 1^TM\bm{\hat U},
	\quad
    \bm{\hat U}(\bm h, \bm P) = \bm U(\bm h, \bm P) + \frac{1}{2}\bm{\hat\beta}(D_1\bm v)^2,
\end{equation}
cf.\ \eqref{eq:total-mod-entropy} and \eqref{eq:mod-entropy}.
Since we need the product and chain rules in \Cref{th:entropy-cons-SK},
we use split forms to guarantee energy conservation of the schemes, see, e.g.\ \cite{fisher2013discretely,gassner2016split,wintermeyer2017entropy,ranocha2017shallow,ranocha2021broad}.
For the shallow water part, we use the same split form as in \cite{ranocha2017shallow}. The dispersive terms are discretized in analogy to \eqref{eq:svaerd-kalisch1}--\eqref{eq:svaerd-kalisch2} using periodic first- and second-derivative SBP operators $D_1$ and $D_2$, and a split form for the $\alpha$-terms in the momentum equation.
The resulting semidiscretization reads
\begin{subequations}
\begin{align}
    \bm h_t ={}& -D_1(\bm h\bm v) + D_1\bm y,\label{eq:semi-SK1}\\
    \begin{split}
    (\bm h\bm v)_t ={}& -\frac{1}{2}\left(D_1(\bm h\bm v^2) + \bm h\bm vD_1\bm v + \bm vD_1(\bm h\bm v)\right) - g\bm hD_1(\bm h + \bm b)\, + \\
    &\frac{1}{2}(D_1(\bm v\bm y) + \bm vD_1\bm y + \bm yD_1\bm v) + D_1(\bm{\hat\beta}D_1\bm v_t)\, + \\
    &\frac{1}{2}D_2(\bm{\hat\gamma}D_1\bm v) + \frac{1}{2}D_1(\bm{\hat\gamma}D_2\bm v),
    \end{split}\label{eq:semi-SK2}
\end{align}
\end{subequations}
where we use the shorthand-notation
\begin{equation*}
    \bm y = \bm{\hat\alpha} D_1(\bm{\hat\alpha}D_1(\bm h + \bm b)) = \bm{\hat\alpha} D_1(\bm{\hat\alpha}D_1\bm\eta).
\end{equation*}
The coefficients $\bm{\hat\alpha}$, $\bm{\hat\beta}$ and $\bm{\hat\gamma}$ are given by
\begin{equation*}
    \bm{\hat\alpha}^2 = \tilde\alpha\sqrt{g\bm D}\bm D^2, \quad\bm{\hat\beta} = \tilde\beta \bm D^3, \quad\bm{\hat\gamma} = \tilde\gamma\sqrt{g\bm D}\bm D^3
\end{equation*}
for the still water depth $\bm D = \eta_0 - \bm b$. Note that we have a time-derivative on the left- and on the right-hand side of the equations in \eqref{eq:semi-SK2}. Therefore, in order to obtain an implementable semidiscretization, we rewrite \eqref{eq:semi-SK2} in primitive variables by using the product rule in time and \eqref{eq:semi-SK1} yielding
\begin{align*}
    \bm\eta_t ={}& D_1(\bm y - (\bm\eta + \bm D)\bm v)\\
    \bm v_t ={}& \left((\bm\eta + \bm D) - D_1\bm{\hat\beta}D_1\right)^{-1}\Bigg(\Bigg.-\frac{1}{2}\left(D_1((\bm\eta + \bm D)\bm v^2) + (\bm\eta + \bm D)\bm vD_1\bm v - \bm vD_1((\bm\eta + \bm D)\bm v)\right)\, - \nonumber\\
    &\qquad\qquad g(\bm\eta + \bm D)D_1\bm\eta + \frac{1}{2}(D_1(\bm v\bm y) -\bm vD_1\bm y + \bm yD_1\bm v) + \frac{1}{2}D_2(\bm{\hat\gamma}D_1\bm v) + \frac{1}{2}D_1(\bm{\hat\gamma}D_2\bm v)\Bigg)\Bigg.,
\end{align*}
where $(\bm\eta + \bm D)$ and $\bm{\hat\beta}$ in the matrix $(\bm\eta + \bm D) - D_1\bm{\hat\beta}D_1$ have to be understood as diagonal matrices.
Notice that the SBP property together with $\bm\eta + \bm D > 0$ and $\bm{\hat\beta} > 0$ implies the
symmetry and positive definiteness of $M((\bm\eta + \bm D) - D_1\bm{\hat\beta}D_1)$
and thus the invertibility of $(\bm\eta + \bm D) - D_1\bm{\hat\beta}D_1$.
\begin{theorem}\label{th:semi-props-SK}
    The semidiscretization \eqref{eq:semi-SK1}--\eqref{eq:semi-SK2} conserves the total mass $\cM$ and total discharge $\cI$ for constant bathymetry, conserves the total modified entropy \eqref{eq:discrete-total-mod-entropy}, and is well-balanced.
\end{theorem}
\begin{proof}
	The proof for the conservation properties of the shallow water part can also be found in \cite[Theorem 1]{wintermeyer2017entropy} with the conversion rules from \cite[Lemma 1]{gassner2016split} (see also \cite{ranocha2017shallow}). For completeness, it is also presented here.

    \Cref{le:per-SBP-prop} item 1 implies the conservation of mass:
    \begin{equation*}
        \frac{\tr d}{\tr dt}\cM(t; \bm h) = \bm 1^TM\bm h_t = -\bm 1^TMD_1(\bm y - \bm h\bm v) = 0.
    \end{equation*}
    Assuming $\bm b = const$, the conservation of the total momentum $\cI$ follows by applying \Cref{le:per-SBP-prop} items 1 and 2:
    \begin{align*}
        \frac{\tr d}{\tr dt}\cI(t; \bm P) ={}& \bm 1^TM(\bm h\bm v)_t \\
        ={}& -\frac{1}{2}\bm 1^TM(D_1(\bm h\bm v^2) + \bm h\bm vD_1\bm v + \bm vD_1(\bm h\bm v)) - g\bm 1^TM(\bm hD_1\bm h) +\\
        &\frac{1}{2}\bm 1^TM(D_1(\bm v\bm y) + \bm vD_1\bm y + \bm yD_1\bm v) + \bm 1^TMD_1(\bm{\hat\beta}D_1\bm v_t)\, + \\
        &\frac{1}{2}\bm 1^TM(D_2(\bm{\hat\gamma} D_1\bm v) + D_1(\bm{\hat\gamma}D_2\bm v))\\
        ={}& -\frac{1}{2}\langle\bm h\bm v, D_1\bm v\rangle_M  -\frac{1}{2}\langle\bm v, D_1(\bm h\bm v)\rangle_M - g\langle\bm h, D_1\bm h\rangle_M +\\
        &\frac{1}{2}\langle\bm v,D_1\bm y\rangle_M + \frac{1}{2}\langle\bm y,D_1\bm v\rangle_M = 0.
    \end{align*}
    For the conservation of the total modified entropy, we begin by looking at the change of the total entropy \eqref{eq:discrete-total-entropy}.
    \begin{align*}
        &\quad \frac{\tr d}{\tr dt}\cE(t; \bm h, \bm P) \\
		&= \bm 1^TM\bm U_t = \langle\bm h_t,g(\bm h + \bm b) - \frac{1}{2}\bm v^2\rangle_M + \langle (\bm h\bm v)_t,\bm v\rangle_M\\
        &= -\langle D_1(\bm h\bm v), g(\bm h + \bm b) - \frac{1}{2}\bm v^2\rangle_M -\frac{1}{2}\left(\langle D_1(\bm h\bm v^2), \bm v\rangle + \langle \bm h\bm v^2,D_1\bm v\rangle_M + \langle \bm v^2, D_1(\bm h\bm v)\rangle_M\right) -\\
        &\quad\;g\langle \bm h\bm v,D_1(\bm h + \bm b)\rangle_M + \langle D_1(\bm{\hat\beta}D_1\bm v_t), \bm v\rangle_M + \frac{1}{2}\langle D_2(\bm{\hat\gamma}D_1\bm v),\bm v\rangle_M + \frac{1}{2}\langle D_1(\bm{\hat\gamma}D_2\bm v),\bm v\rangle_M\, + \\
        &\quad\;\langle D_1\bm y, g(\bm h + \bm b) - \frac{1}{2}\bm v^2\rangle_M + \frac{1}{2}\langle D_1(\bm v\bm y),\bm v\rangle_M + \frac{1}{2}\langle\bm vD_1\bm y,\bm v\rangle_M + \frac{1}{2}\langle\bm yD_1\bm v,\bm v\rangle_M.
    \end{align*}
    The terms coming from the SWEs cancel due to the split form:
    \begin{equation*}
    	-g\langle D_1(\bm h\bm v),\bm h + \bm b\rangle_M - g\langle \bm h\bm v,D_1(\bm h + \bm b)\rangle_M - \frac{1}{2}\langle D_1(\bm h\bm v^2), \bm v\rangle_M - \frac{1}{2}\langle \bm h\bm v^2,D_1\bm v\rangle_M = 0.
    \end{equation*}
    Next, we consider the $\gamma$-terms:
    \begin{equation*}
        \frac{1}{2}\langle D_2(\bm{\hat\gamma}D_1\bm v),\bm v\rangle_M + \frac{1}{2}\langle D_1(\bm{\hat\gamma}D_2\bm v),\bm v\rangle_M = \frac{1}{2}(\langle\bm{\hat\gamma}D_1\bm v, D_2\bm v\rangle_M - \langle\bm{\hat\gamma} D_2\bm v, D_1\bm v\rangle_M) = 0,
    \end{equation*}
    where we use the SBP property and symmetry of the diagonal norm matrix $M$. Thanks to the split form of the $\alpha$-terms, they also cancel:
    \begin{align*}
        &\langle D_1\bm y, g(\bm h + \bm b) - \frac{1}{2}\bm v^2\rangle_M + \frac{1}{2}\langle D_1(\bm v\bm y),\bm v\rangle_M + \frac{1}{2}\langle\bm vD_1\bm y,\bm v\rangle_M + \frac{1}{2}\langle\bm yD_1\bm v,\bm v\rangle_M\\
        &= -g\langle\bm{\hat\alpha}D_1(\bm{\hat\alpha}D_1(\bm h + \bm b)),D_1(\bm h + \bm b)\rangle_M - \frac{1}{2}\langle D_1\bm y,\bm v^2\rangle_M - \frac{1}{2}\langle\bm v\bm y,D_1\bm v\rangle_M\, + \\
        &\quad\;\frac{1}{2}\langle D_1\bm y,\bm v^2\rangle_M + \frac{1}{2}\langle D_1\bm v,\bm v\bm y\rangle_M \\
        &= -g\langle D_1\bm z,\bm z\rangle_M = 0.
    \end{align*}
    In the last step we use \Cref{le:per-SBP-prop} item 2 for $\bm z = \bm{\hat\alpha}D_1(\bm h + \bm b)$.
    Hence, we are left with
    \begin{align*}
        \bm 1^TM\bm U_t &= \langle D_1(\bm{\hat\beta}D_1\bm v_t),\bm v\rangle_M = -\langle\bm{\hat\beta}D_1\bm v_t, D_1\bm v\rangle_M = -\bm 1^TM\bm{\hat\beta}D_1\bm vD_1\bm v_t \\
		&= -\frac{1}{2}\bm 1^TM\bm{\hat\beta}((D_1\bm v)^2)_t,
    \end{align*}
    which is equivalent to
    \begin{equation*}
        \frac{\tr d}{\tr dt}\hat\cE(t; \bm h, \bm P) = \bm 1^TM\bm{\hat U}_t = \bm 1^TM\left(\bm U + \frac{1}{2}\bm{\hat\beta}(D_1\bm v^2)\right)_t = 0,
    \end{equation*}
    i.e., conservation of the total modified entropy.\\
    Since each term in \eqref{eq:semi-SK1}--\eqref{eq:semi-SK2} either contains a factor $\bm v$ or a factor $D_1(\bm h + \bm b)$ consistency of $D_1$ implies well-balancedness of the scheme.
\end{proof}
As an alternative semidiscretization, we can again use upwind operators to avoid using multiple central operators in the second- and third-order derivatives leading to wide-stencil operators. One possibility is the semidiscretization
\begin{subequations}
\begin{align}
    \bm h_t &= -D_1(\bm h\bm v) + D_{1,-}\bm y,\label{eq:semi-up-SK1}\\
    \begin{split}
    (\bm h\bm v)_t &= -\frac{1}{2}\left(D_1(\bm h\bm v^2) + \bm h\bm vD_1\bm v + \bm vD_1(\bm h\bm v)\right) - g\bm hD_1(\bm h + \bm b)\, + \\
    &\quad\;\frac{1}{2}(D_{1,-}(\bm v\bm y) + \bm vD_{1,-}\bm y + \bm yD_{1,+}\bm v) + D_{1,+}(\bm{\hat\beta}D_{1,-}\bm v_t)\, + \\
    &\quad\;\frac{1}{2}D_2(\bm{\hat\gamma}D_1\bm v) + \frac{1}{2}D_1(\bm{\hat\gamma}D_2\bm v),
    \end{split}\label{eq:semi-up-SK2}
\end{align}
\end{subequations}
where
$\bm y = \bm{\hat\alpha} D_{1,-}(\bm{\hat\alpha}D_{1,+}(\bm h + \bm b)) = \bm{\hat\alpha} D_{1,-}(\bm{\hat\alpha}D_{1,+}\bm\eta)$.
\begin{theorem}
	The semidiscretization \eqref{eq:semi-up-SK1}--\eqref{eq:semi-up-SK2} conserves the total mass $\cM$ and total discharge $\cI$ for constant bathymetry and is entropy-stable in terms of the modified entropy \eqref{eq:discrete-total-mod-entropy}, i.e.,
	\begin{equation*}
		\frac{\tr d}{\tr dt}\hat\cE(t; \bm h, \bm P) \le 0.
	\end{equation*}
	Moreover, it is well-balanced.
\end{theorem}
\begin{proof}
	The proof can be conducted in the same way as before in \Cref{th:semi-props-SK}, but now we obtain $-\langle D_{1,-}\bm z,\bm z\rangle\le 0$ for $\bm z = \bm{\hat\alpha}D_{1,+}(\bm h + \bm b)$ from the $\alpha$-terms, yielding the dissipation according to \Cref{le:per-SBP-prop} item 3. Note that the third-derivative operators in the $\alpha$-terms of this semidiscretization are biased towards the negative direction, which leads to the dissipation.
\end{proof}
Similarly, a conservative semidiscretization using upwind operators can be found.

Compared to the discretization given in \cite{svard2023novel}, the presented semidiscretizations have some advantages. First, one can easily obtain schemes of arbitrary order as we just need to pick SBP operators of appropriate order. Second, the discretization follows closely the continuous equations, so theoretical results of the equations transfer to the semidiscretization in a straightforward manner.
Especially, entropy stability can be obtained without the need to introduce artificial diffusion as it is done in \cite{svard2023novel}. Last, the general framework of SBP operators allows different types of methods, and we can use any time-stepping scheme to solve the semidiscrete system. Especially combined with relaxation methods, as described in the following section, fully-discrete entropy conservation can be ensured.

In the case of $\tilde\alpha = \tilde\gamma = 0$, we consider reflecting boundary conditions,
where we have $v = 0$ at the boundary of the domain. We only consider this case as for general
$\tilde\alpha$ and $\tilde\gamma$, the equations do not satisfy an entropy bound. Note that
the additional condition $\eta_x = 0$ on the boundary is obsolete as there are no higher order
derivatives in the mass equation. Under these conditions, we can formulate the semidiscretization
\begin{subequations}
	\begin{align}
		\bm h_t &= -D_1(\bm h\bm v),\label{eq:semi-SK1-reflecting}\\
		\begin{split}
			(\bm h\bm v)_t &= -\frac{1}{2}\left(D_1(\bm h\bm v^2) + \bm h\bm vD_1\bm v + \bm vD_1(\bm h\bm v)\right) - g\bm hD_1(\bm h + \bm b)\, + \\
			&\quad\; P_DD_1(\bm{\hat\beta}D_1\bm v_t),
		\end{split}\label{eq:semi-SK2-reflecting}
	\end{align}
\end{subequations}
where, again, $P_D = \diag(0, 1, \ldots, 1, 0)$. We explicitly set $v_1 = v_N = 0$ to
satisfy the Dirichlet boundary conditions strongly.

\begin{theorem}
	The semidiscretization \eqref{eq:semi-SK1-reflecting}--\eqref{eq:semi-SK2-reflecting}
	conserves the total mass $\cM$, conserves the total modified entropy
	\eqref{eq:discrete-total-mod-entropy}, and is well-balanced.
\end{theorem}
\begin{proof}
	The proof can be obtained in the same way as in \Cref{th:semi-props-SK}, where for
	the entropy condition we note that $P_D\bm v = \bm v$ and $\bm{e}_{L/R}^T\bm v = 0$.
\end{proof}
An entropy-conservative upwind discretization can be obtained by replacing
$P_DD_1(\bm{\hat\beta}D_1\bm v_t)$ by $P_DD_{1,+}(\bm{\hat\beta}D_{1,-}\bm v_t)$.

\subsection{Relaxation method}\label{sect:time-integration}

The semidiscretizations from the previous sections are given as
ordinary differential equations (ODEs)
that conserve ($\frac{\tr d}{\tr dt}J(u) = 0$)
or dissipate ($\frac{\tr d}{\tr dt}J(u) \le 0$)
some nonlinear, convex, and differentiable functional $J(u)$ for all
solutions $u$, where we identify $J$ as the
total energy $\cE$ or modified entropy $\hat\cE$.
To preserve the energy conservation or dissipation for the
fully-discrete scheme, we apply relaxation methods
\cite{ketcheson2019relaxation,ranocha2020relaxation,ranocha2020general},
resulting in \emph{conservative} time-stepping schemes satisfying
$J(u^{n + 1}) = J(u^n) = J(u^0)$ and \emph{dissipative}
(also known as \emph{monotone} or \emph{strongly stable}
\cite[Definition 2.2]{ranocha2021strong}) methods satisfying
$J(u^{n + 1}) \le J(u^n)$.
In our simulations we take explicit Runge-Kutta methods as baseline
one-step methods, in which case the relaxation method is also
referred to as \emph{relaxation Runge-Kutta} (RRK) method,
cf.\ \cite[Section 2.2]{ketcheson2019relaxation}.

\section{Results}\label{sect:results}
The numerical methods are implemented in the Julia \cite{bezanson2015julia} package DispersiveShallowWater.jl \cite{lampert2023dispersive}, which is partly based on the numerical simulation framework Trixi.jl \cite{ranocha2022adaptive,schlottkelakemper2020trixi,schlottkelakemper2021purely}.
The SBP operators are taken from SummationByPartsOperators.jl \cite{ranocha2021sbp}. Time integration is performed using OrdinaryDiffEq.jl \cite{rackauckas2017differentialequations}.

The adaptive explicit Runge-Kutta method by Tsitouras \cite{tsitouras2011rungekutta} is used for the integration in time (\texttt{Tsit5} in OrdinaryDiffEq.jl). If not stated otherwise, we use error-based adaptive step size control with absolute and relative tolerances of $10^{-14}$ for the convergence tests to limit the error to the spatial discretization and tolerances of $10^{-7}$ for the remaining experiments.

The source code to reproduce the figures presented in this paper can be found online \cite{lampert2024structureRepro}.

\subsection{Solitary wave solution of the BBM-BBM equations}
We use the solution \eqref{eq:BBMBBM-soliton1}--\eqref{eq:BBMBBM-soliton2} of the BBM-BBM equations with periodic
boundary conditions, a constant still water depth $D \equiv 2$, and $g = 9.81$ to test the convergence properties of the scheme.
The initial condition follows from \eqref{eq:BBMBBM-soliton1}--\eqref{eq:BBMBBM-soliton2} by taking $t = 0$. We use the
semidiscretization \eqref{eq:semi-BBM-BBM-const1}--\eqref{eq:semi-BBM-BBM-const2} to compare the results obtained using the RRK method
in time with the baseline method without relaxation. The $L^2$-errors of the solutions at $t = 10$ for $N$
points in space (degrees of freedom, DOFs) and different orders of accuracy $p$ are depicted in \Cref{fig:order-soliton} for narrow-stencil finite difference
SBP operators.

\begin{figure}[htbp]
    \centering
    \includegraphics[width=0.9\textwidth]{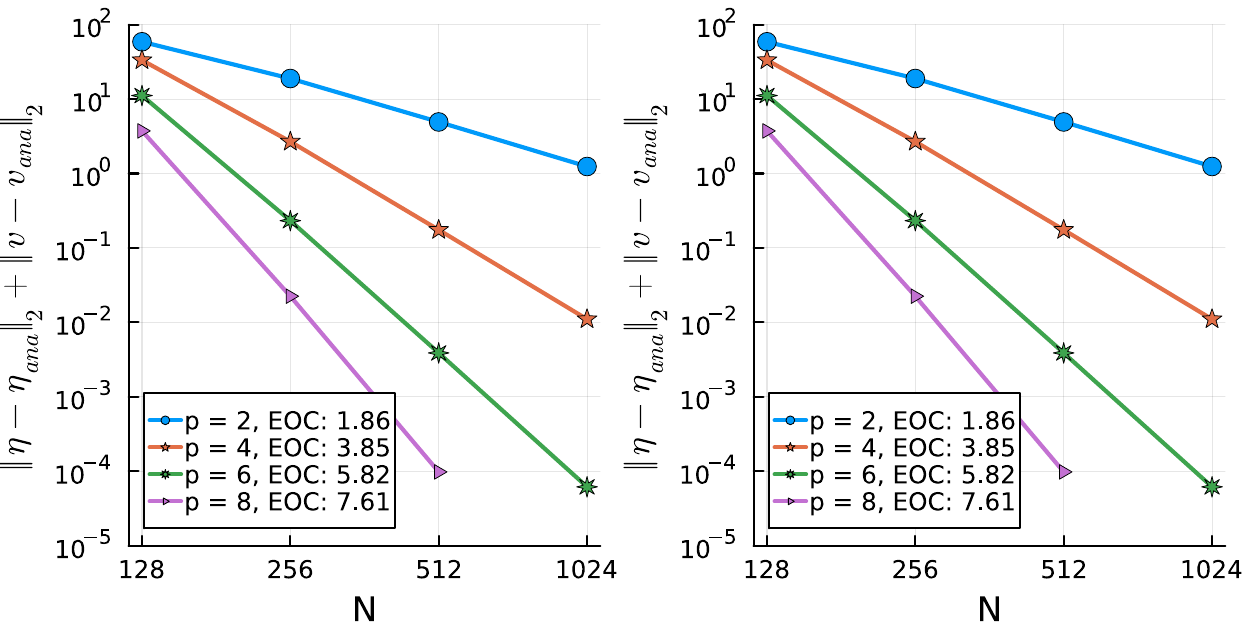}
    \caption{Experimental order of convergence for the solitary wave solution of the BBM-BBM equation, left: baseline without relaxation, right: with relaxation}
    \label{fig:order-soliton}
\end{figure}

With both approaches, the observed spatial experimental order of convergence (EOC) approximately corresponds to the order of accuracy of the SBP operators. A more extensive convergence study, also considering different types of solvers, is presented in \cite{ranocha2021broad}.

\begin{figure}[htbp]
    \centering
    \includegraphics[width=0.9\textwidth]{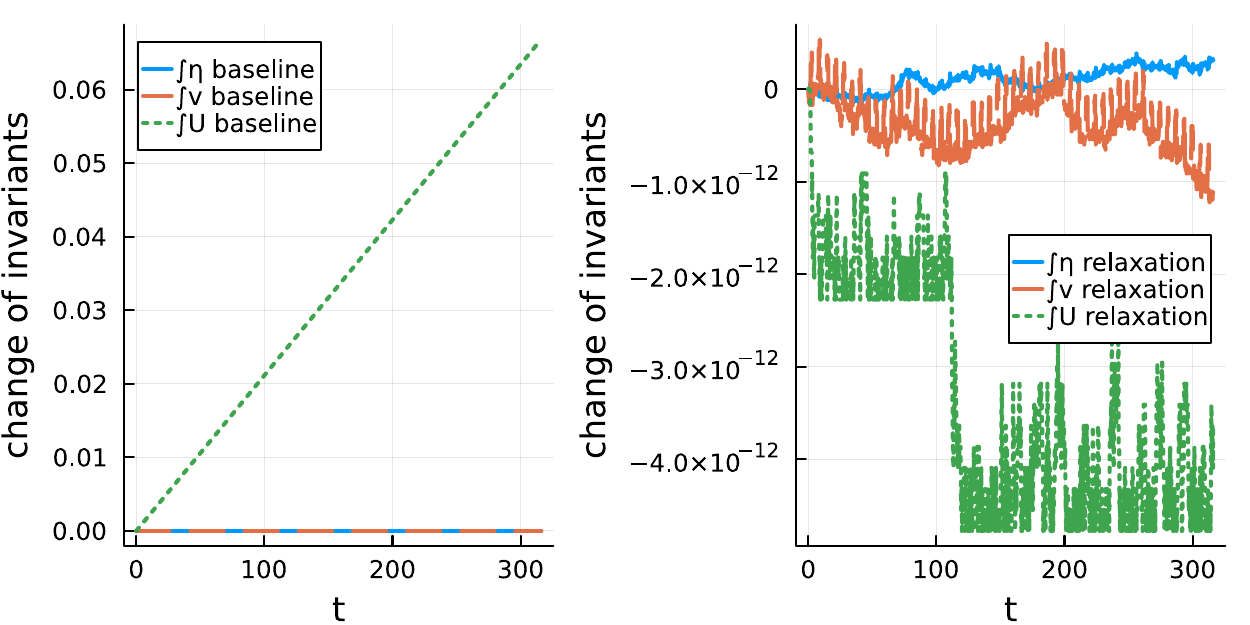}
    \caption{Conservation of linear and nonlinear invariants for the solitary wave solution of the BBM-BBM equations, left: baseline without relaxation, right: with relaxation}
    \label{fig:invariants}
\end{figure}

The spatial error dominates in \Cref{fig:order-soliton}, such that both methods exhibit similar errors. Therefore, to further analyze the impact of the relaxation method, we solve the same setup using SBP operators with order of accuracy $p = 8$ and $N = 512$ points in space (DOFs), but run the simulation up to a final time of $t \approx 316 $, which corresponds to 50 periods.
\Cref{fig:invariants} shows the temporal development of the change of the linear invariants and the nonlinear invariant $\cE$.
In particular, it can be seen that the total energy grows linearly without using relaxation, whereas it is constant up to machine precision when the RRK method is used. Both methods conserve the total mass $\cM$ exactly, and the total velocity $\cV$ up to $11$ digits. That the relaxation method can lead to an improved approximation is also demonstrated by the error growth, which is plotted up to $t = 100$ in \Cref{fig:errors}.
It shows that the relaxation solution is able to track the analytical solution accurately for a much longer time period than the baseline method.

\begin{figure}[htbp]
    \centering
    \includegraphics[width=0.9\textwidth]{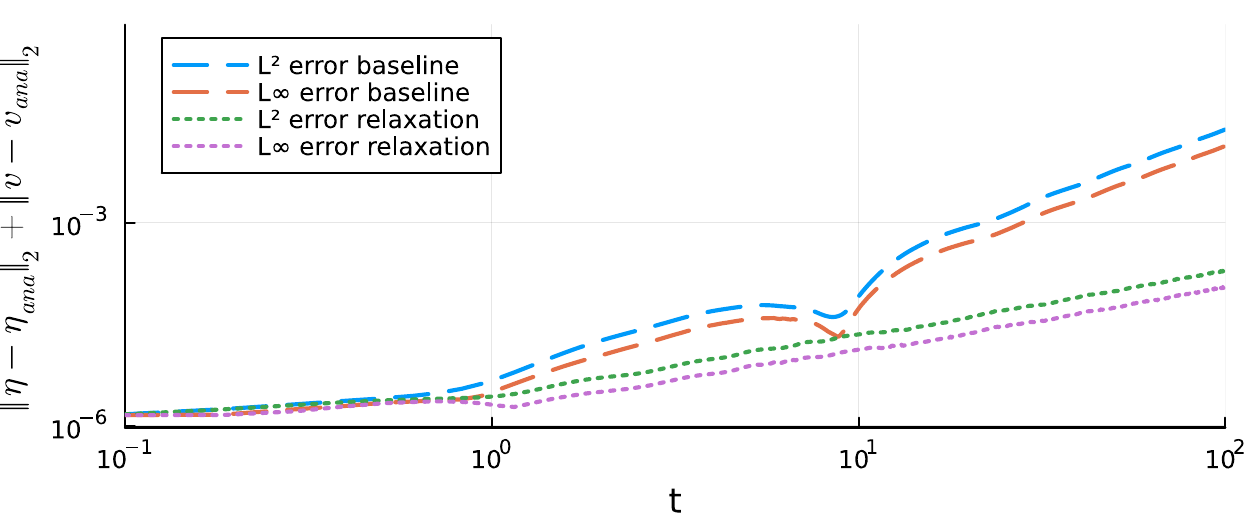}
    \caption{Evolution of $L^2$- and $L^\infty$-errors for baseline and relaxation method applied to the solitary wave solution of the BBM-BBM equations}
    \label{fig:errors}
\end{figure}

The time step chosen by the adaptive Runge-Kutta method is approximately $10^{-2}$ and independent of the order of accuracy and number of grid points. The relaxation parameter $\gamma$ is of the size $1 + 10^{-8}$ over the whole time span,
i.e., the relaxation method yields slightly bigger time steps, which, in turn,
leads to a lower number of total time steps. Note, however, that in the case of error-based step size control
with a first-same-as-last (FSAL) method, the right-hand side of the relaxed solution needs to be
re-evaluated at the beginning of the next time step. A remedy for this problem is proposed in \cite{bleecke2023step}.

\begin{figure}[htbp]
    \centering
    \includegraphics[width=0.9\textwidth]{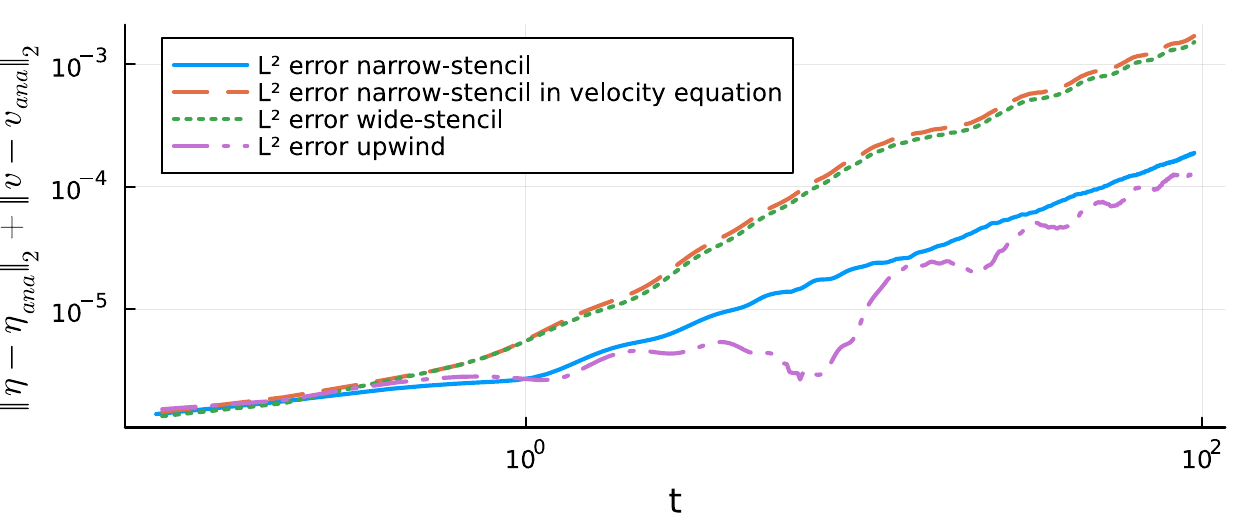}
    \caption{Evolution of $L^2$-errors for different stencils in finite difference semidiscretizations of the BBM-BBM equations applied to
            the solitary wave solution with constant bathymetry}
    \label{fig:errors-stencils}
\end{figure}

Next, we compare the semidiscretizations presented for the BBM-BBM equations with variable bathymetry. Specifically, we look at three semidiscretizations: First, \eqref{eq:semi-BBM-BBM1}--\eqref{eq:semi-BBM-BBM2} with narrow-stencil second-derivative operator in the equation \eqref{eq:semi-BBM-BBM2} for the velocity. Second, the semidiscretization \eqref{eq:semi-BBM-BBM1}--\eqref{eq:semi-BBM-BBM2} with wide-stencil operator $D_2 = D_1^2$ and third, the upwind semidiscretization \eqref{eq:semi-up-BBM-BBM1}--\eqref{eq:semi-up-BBM-BBM2}. The first is not energy-conservative, while the second and the third are energy-conserving. As baseline method, we take again the semidiscretization \eqref{eq:semi-BBM-BBM-const1}--\eqref{eq:semi-BBM-BBM-const2}, which uses narrow-stencil operators in both equations. The growth of the $L^2$-error for a time span of 15 periods is presented in \Cref{fig:errors-stencils}, again using an order of accuracy of 8.
Clearly, using one or two wide-stencil operators leads to significantly larger errors, however the energy-conservative semidiscretization that uses the wide-stencil operator $D_2 = D_1^2$ in both equations \eqref{eq:semi-BBM-BBM1}--\eqref{eq:semi-BBM-BBM2} is slightly more accurate than the semidiscretization that uses a narrow-stencil operator in the second equation. Upwind operators yield the smallest error, which demonstrates that we do not lose accuracy using the general scheme designed for the BBM-BBM equations with variable bathymetry compared to the scheme designed for the BBM-BBM equations with constant bathymetry.

To sum up, these experiments demonstrate that the presented semidiscretization of the BBM-BBM equations with constant bathymetry is able to capture the solitary wave. The linear and nonlinear invariants are conserved as expected from the analysis. Furthermore, it is exemplified that a fully-discrete energy-conserving numerical scheme can stabilize the method and can lead to higher accuracy of the numerical solution. Finally, the energy-conservative upwind semidiscretization for the general BBM-BBM equations with (possibly) variable bathymetry is able to produce similarly accurate results as the semidiscretization for the constant case.

\subsection{Convergence test}
Since there are no known analytical solutions for the BBM-BBM equations with variable bathymetry or the equations of Svärd and Kalisch, we construct manufactured solutions to validate the schemes. For both sets of equations we use the solution
\begin{equation*}
	\eta(t, x) = \mr e^t\cos(2\pi(x - 2t)), \quad v(t, x) = \mr e^{t/2}\sin(2\pi(x - t/2)),
\end{equation*}
with $x\in[0,1], t\in[0,1]$ in a periodic domain, and the bathymetry
\begin{equation*}
	b(x) = -5 - 2\cos(2\pi x)
\end{equation*}
to construct source terms for the equations.

\begin{figure}[htbp]
    \centering
	\includegraphics[width=\textwidth]{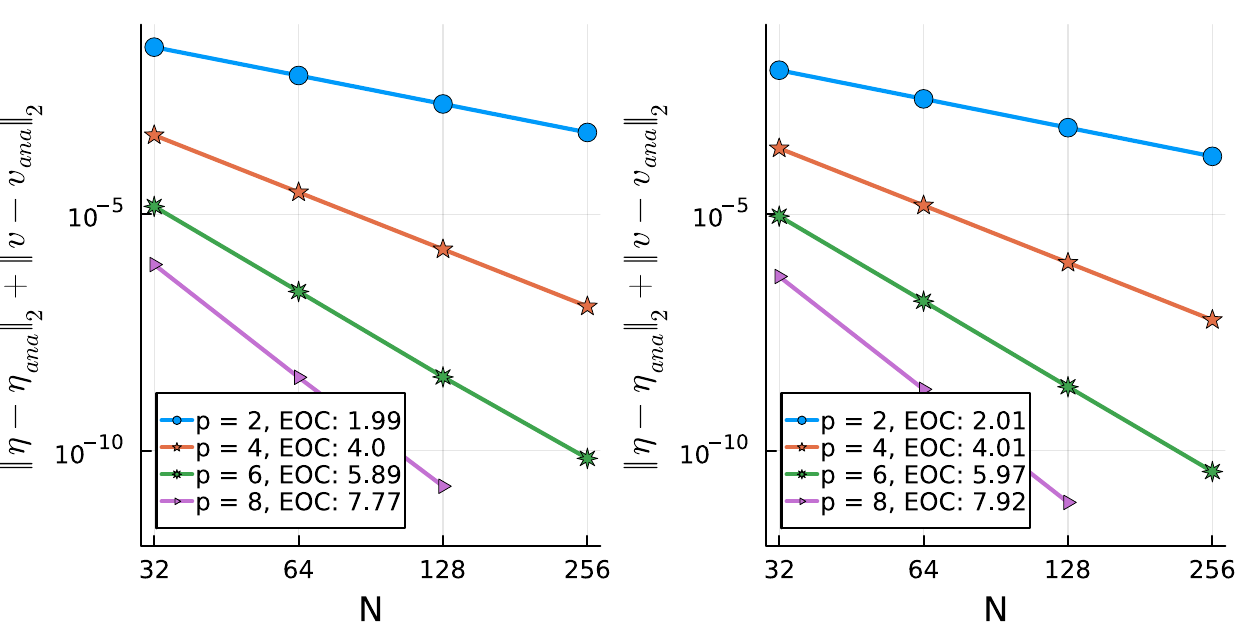}
	\caption{Experimental order of convergence for manufactured solution with variable bathymetry (left: BBM-BBM equations, right: Svärd-Kalisch equations)}
	\label{fig:order-manufactured}
\end{figure}

\Cref{fig:order-manufactured} shows the EOC for the schemes
with upwind operators and periodic boundary conditions for the two equations. Again, the
EOC matches the order of accuracy of the involved SBP operators.
Similar results are also obtained for the corresponding wide-stencil semidiscretizations.

\begin{figure}[htbp]
	\centering
	\includegraphics[width=0.95\textwidth]{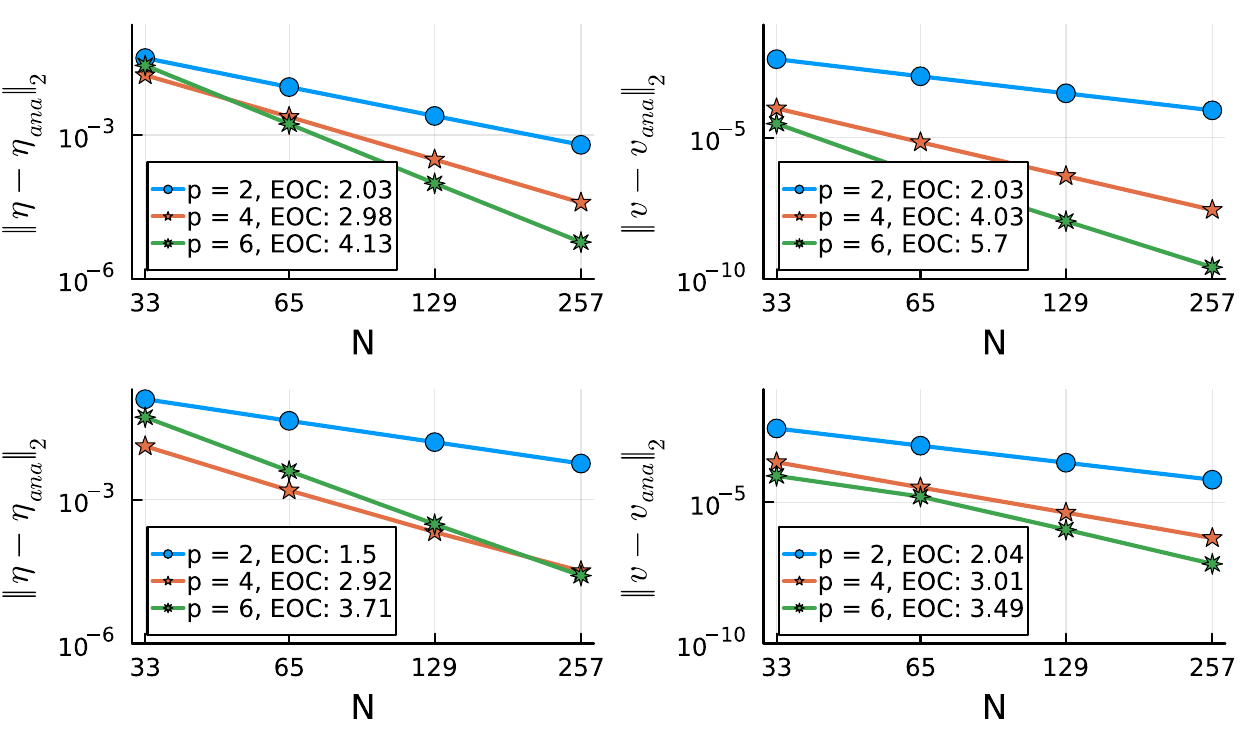}
	\caption{Experimental order of convergence for manufactured solution with reflecting boundary conditions (left: convergence in $\eta$, right: convergence in $v$, upper: BBM-BBM equations, lower: Svärd-Kalisch equations)}
	\label{fig:order-manufactured-reflecting}
\end{figure}

To verify the semidiscretizations with solid wall boundary conditions, we use the manufactured solution
\begin{equation*}
	\eta(t, x) = \mr e^{2t}\cos(\pi x), \quad v(t, x) = \mr e^tx\sin(\pi x),
\end{equation*}
and use the same bathymetry function as above. The $L^2$-errors for both $\eta$ and $v$ and both equations are depicted in
\Cref{fig:order-manufactured-reflecting}, where we use the SBP operators from \cite{mattsson2004summation}. Due to the
reduced order at the boundaries of the non-periodic SBP operators, we also observe reduced experimental orders of convergence
for $p\in\{4, 6\}$. Additionally, as already found in \cite{ranocha2021broad}, the order depends on the parity of the number
of nodes $N$. Choosing an even number of discretization points reduces the accuracy significantly.

\subsection{Well-balanced verification}
To verify the well-balanced property of our schemes for both the BBM-BBM equations and the Svärd-Kalisch equations, we solve a lake-at-rest test case with discontinuous bathymetry
\begin{equation*}
	b(x) = \begin{cases}
		1.5 + 0.5\sin(2\pi x), &\text{ for } 0.5 \le x \le 0.75,\\
		1.0, &\text{ else},
	\end{cases}
\end{equation*}
which has a jump of height $0.5$ at $x = 0.5$.
We consider the domain $[-1, 1]$ and set the initial conditions as $\eta(0, x) = 2$ and $v(0, x) = 0$.
We use 200 degrees of freedom in space, central SBP operators of orders of 2, 4, and 6 and
let the simulation run up to a time of $t = 10$. Here, we take a fixed time step for the time integration. Due to the high stiffness of
the Svärd-Kalisch equations we take a rather small time step of $\Delta t = 2\cdot 10^{-4}$, while for the BBM-BBM equations a bigger step
size of $\Delta t = 0.5$ suffices. The $L^2$-errors for periodic boundary conditions can be found in \Cref{table:well-balanced-errors} and
are of the order of machine precision. Similar results are obtained with reflecting boundary conditions.
We note that the exact preservation of steady-state solutions is achieved by using the formulation in primitive variables $\eta$ and $v$.

\begin{table}[!htb]
	\caption{$L^2$-error for lake-at-rest test case for different orders of accuracy $p$}
	\label{table:well-balanced-errors}
	\centering
	\begin{tabular}{ccc @{\hskip 5em} ccc}
		\toprule
		\multicolumn{3}{c}{BBM-BBM equations} & \multicolumn{3}{c}{Svärd-Kalisch equations} \\
		$p$ & $\eta$ & v & $p$ & $\eta$ & v \\\midrule
		2 & 0.00e+00 & 0.00e+00 & 2 & 0.00e+00 & 0.00e+00 \\
		4 & 3.95e-15 & 8.62e-15 & 4 & 5.28e-14 & 3.22e-15 \\
		6 & 9.00e-15 & 3.04e-15 & 6 & 3.96e-14 & 8.38e-15 \\\bottomrule
	\end{tabular}
\end{table}

\subsection{Conservation for reflecting boundary conditions}

To verify the conservation properties of the presented semidiscretizations with reflecting wall boundary conditions, we choose an initial condition with a bump in the water height, initially still water, and a varying bathymetry, i.e.,
\begin{equation*}
	\eta(t = 0, x) = 1 + \mr{e}^{-50x^2}, \quad
	v(t = 0, x) = 0, \quad
	b(x) = 0.3 \cos(\pi x).
\end{equation*}
We use $N = 512$ points in the domain $[-1, 1]$ and integrate up to a time $t = 1$. The initial bump splits into one left-traveling and one right-traveling wave that interact with the boundary at $t \approx 0.25$. Again, we compare the conservation of the invariants $\cM$ and $\cE$ with and without using the RRK, cf. \Cref{fig:invariants-reflecting-BBM-BBM} for the BBM-BBM equations and \Cref{fig:invariants-reflecting-SK} for the Svärd-Kalisch equations.

\begin{figure}[htbp]
	\centering
	\includegraphics[width=0.9\textwidth]{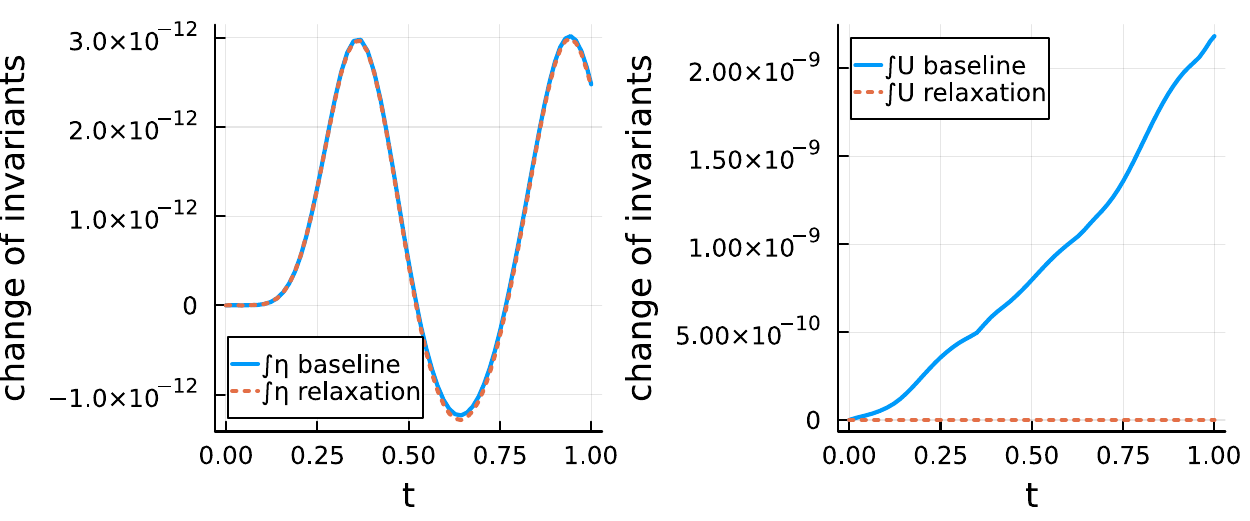}
	\caption{Conservation of linear and nonlinear invariants for BBM-BBM equations with reflecting boundary conditions with and without relaxation, left: conservation of $\cM$, right: conservation of $\cE$}
	\label{fig:invariants-reflecting-BBM-BBM}
\end{figure}

\begin{figure}[htbp]
	\centering
	\includegraphics[width=0.9\textwidth]{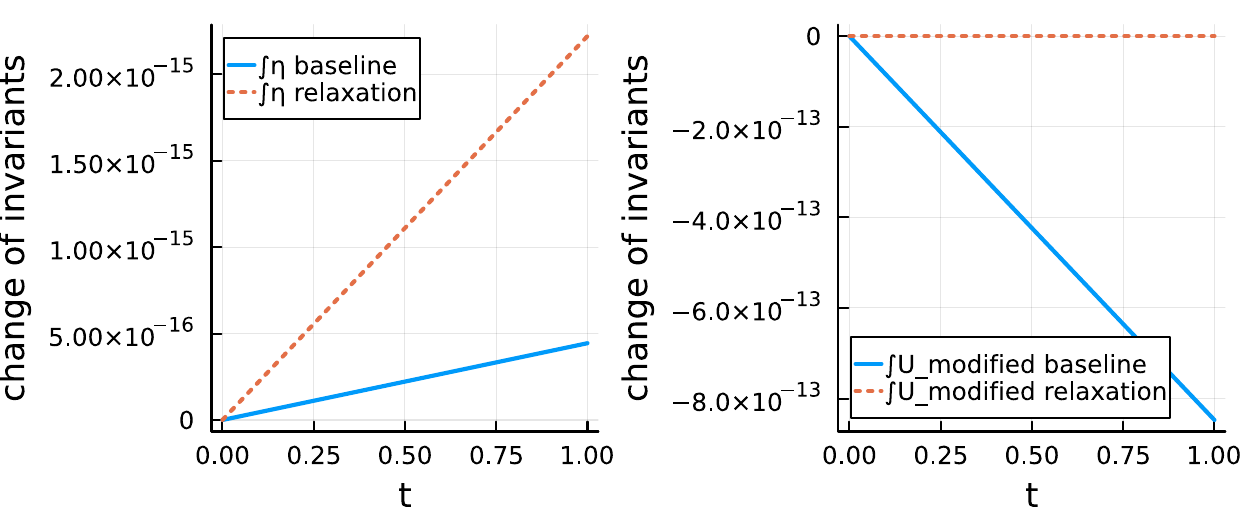}
	\caption{Conservation of linear and nonlinear invariants for Svärd-Kalisch equations with reflecting boundary conditions with and without relaxation, left: conservation of $\cM$, right: conservation of $\cE$}
	\label{fig:invariants-reflecting-SK}
\end{figure}

The baseline method already conserves the total energy up to 9 respectively 13 digits. The small gain in the total energy is, again, due to the time integration, which is not energy-conserving. Applying the relaxation Runge-Kutta method a fully-discrete energy-conserving method is obtained. Both versions conserve the total mass up to machine precision.

\subsection{Traveling wave}
To compare the different dispersive behaviors of the equations, we simulate a traveling wave solution of the Euler equations for different wavenumbers $k$ over a flat bottom with periodic boundaries. To this end, we use the dispersion relation of the Euler equations \eqref{eq:disp-rel-Euler}, similarly to \cite{svard2025novel}, and set
\begin{align*}
	\eta(t = 0, x) &= \eta_0 + A\cos(kx), &\quad
	v(t = 0, x) &= \sqrt{\frac{g}{k}\tanh(kh_0)}\frac{\eta - \eta_0}{h_0},
\end{align*}
where $\omega = \sqrt{gk\tanh(kh_0)}$. We take $\eta_0 = 0$, $h_0 = 0.8$, $D = h_0$, $g = 9.81$, $A = 0.02$, and $k\in\{0.8, 5, 15\}$.
\begin{figure}[ht]
	\centering
	\includegraphics[width=0.9\textwidth]{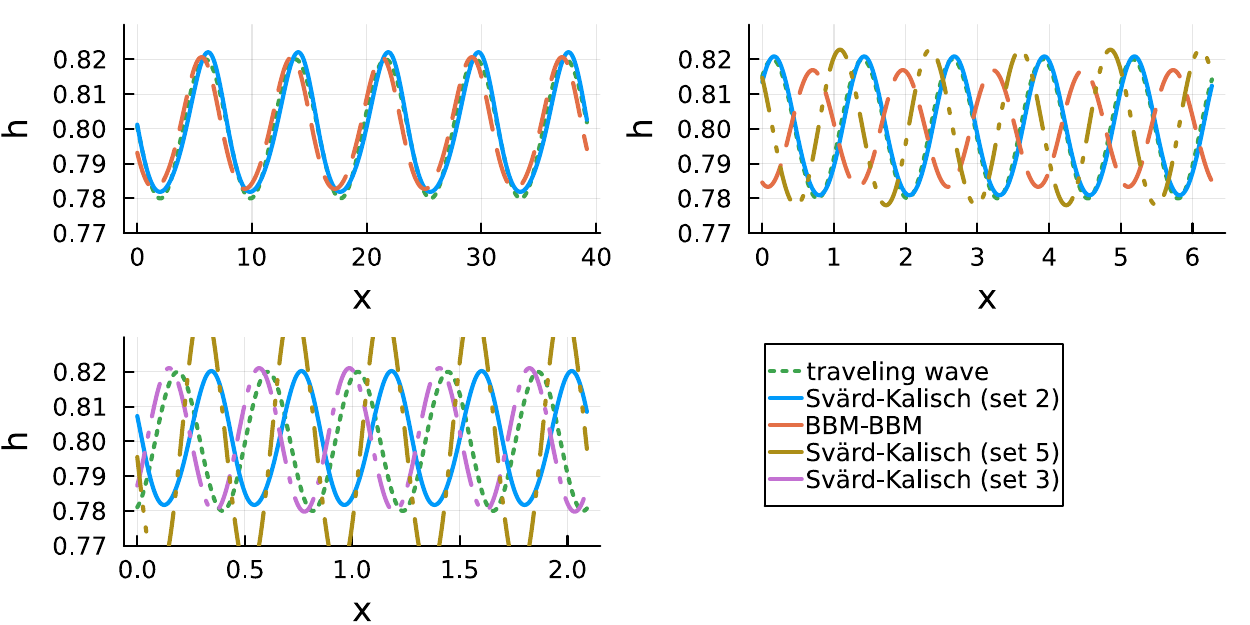}
	\caption{Water height for different models for $k = 0.8$ and $t = 50$ (top left), $k = 5$ and $t = 1$ (top right), and $k = 15$ and $t = 0.75$ (lower left)}
	\label{fig:waterheight-traveling-wave}
\end{figure}
Since the initial condition is chosen to be a traveling wave for the linearized Euler equations, the reference solution is given by
\begin{equation*}
	\eta(t, x) = \eta_0 + A\cos(kx - \omega t).
\end{equation*}
The periodic domain is taken such that $5$ waves are modeled. We use central finite difference SBP operators with $N = 512$ nodes in space. For $k = 0.8$, we expect from \Cref{fig:disp-rels} that both the BBM-BBM and Svärd-Kalisch equations (especially set~2) approximate the wave speed fairly well. The top left subplot of \Cref{fig:waterheight-traveling-wave} verifies that the phase error for both models is indeed quite small, even for $t = 50$. As expected from the dispersion relations, the phase velocity is almost exact for set~2 of the Svärd-Kalisch equations and slightly too small for the BBM-BBM equations ($c_\text{Euler}\approx 2.6319$, $c_\text{BBM-BBM}\approx 2.6224$, and $c_\text{set~2}\approx 2.6316$). In addition, we can see a small error in the amplitude of the solution to the Svärd-Kalisch equations.

Taking waves with much smaller wavelengths, however, leads to a significant phase (and amplitude) error of the BBM-BBM equations already at $t = 1$ as can be seen in the top right part of \Cref{fig:waterheight-traveling-wave} for $k = 5$. At the same time, set~2 of the Svärd-Kalisch equations is still able to accurately capture the speed and also amplitude of the waves. Set~5 lies in between set~2 and the BBM-BBM equations. This property is explained by the improved dispersive behavior of set~2 for the Svärd-Kalisch equations for this range of wavenumbers, see \Cref{fig:disp-rels}.
When we take an even higher wavenumber, $k = 15$, the dispersion relation of set~2 deviates from the one of the Euler equations significantly, but set~3 is still able to describe the dispersive behavior reasonably well. The errors for set~5 become even larger. This is demonstrated in the lower left subplot of \Cref{fig:waterheight-traveling-wave}.

\subsection{Comparison with experimental data}

To validate the discretizations of the different models, we compare their behavior for a setup with experimental data \cite{dingemans1994comparison,dingemans1997water}. The experiment of Dingemans \cite{dingemans1994comparison} used a wave maker situated at $x = 0$ to produce waves with amplitude $A = 0.02$ that pass a trapezoidal bathymetry. The bottom is flat up to $x = 11.01$, then increases linearly up to $x = 23.04$ with a maximum height of 0.6 and finally decreases again linearly from $x = 27.04$ to $x = 33.07$, see also \cite{beji1993experimental}. To simulate the waves generated by the wave maker, we use the dispersion relation of the Euler equations, similarly to \cite{svard2023novel} (see also previous section). We pick the gravitational constant as $g = 9.81$. The reference water height $h_0 = \eta_0$ is given by 0.8, and the initial condition by
\begin{align*}
    \eta(t = 0, x) &= \begin{cases}
        \eta_0 + A\cos(k(x - \tilde x)),&\text{ if } -34.5\frac{\pi}{k} < x - \tilde x < -4.5\frac{\pi}{k},\\
        \eta_0,&\text{ else},
    \end{cases}\\
    v(t = 0, x) &= \sqrt{\frac{g}{k}\tanh(kh_0)}\frac{\eta - \eta_0}{h_0},
\end{align*}
with wavenumber $k\approx 0.84$ satisfying the dispersion relation
$\omega^2 = gk\tanh(kh_0)$
for the angular frequency $\omega = \frac{2\pi}{2.02\sqrt{2}}$. The value of $\tilde x$ is calibrated such that the time shift between the initial condition and the waves from the experimental data is compensated. For the BBM-BBM equations we take $\tilde x = 2.7$ and for the Svärd-Kalisch equations $\tilde x = 2.2$. The choice of the initial conditions for the total water height corresponds to 15 waves that transition continuously into a constant state. We run the setup in a periodic domain $x\in[-138, 46]$ up to a time $t = 70$. For this example, we apply periodic boundaries. Compared to other benchmarks from the literature \cite{madsen1996boussinesq,israwi2023equations,kazolea2023full}, we extended the domain to the left to avoid any initial interaction between the waves and the topography.

\begin{figure}[htbp]
	\centering
	\includegraphics[width=1.0\textwidth]{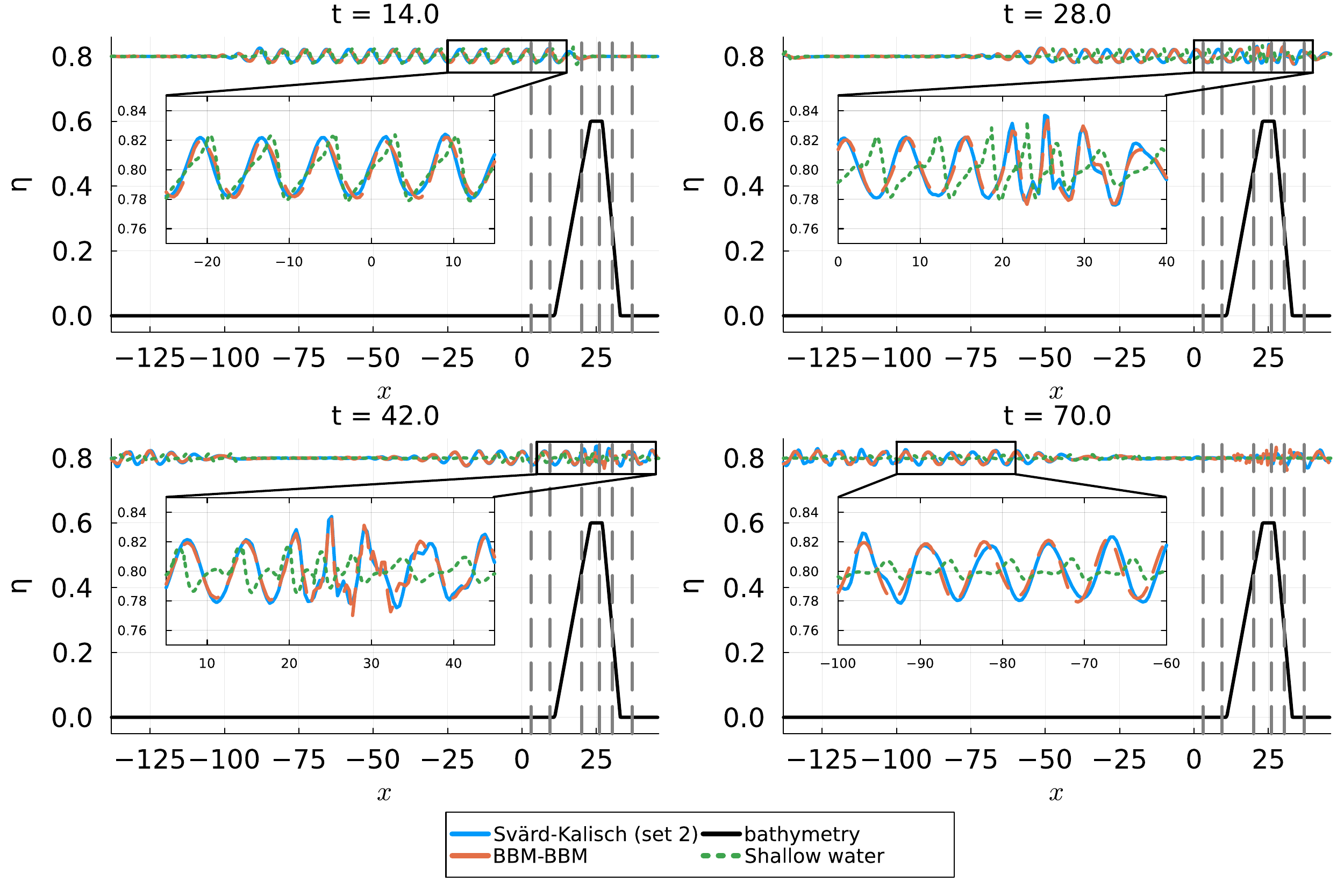}
	\caption{Water height at different points in time and for different models}
	\label{fig:waterheight-over-time-dingemans}
\end{figure}

\Cref{fig:waterheight-over-time-dingemans} shows the computed solutions for the SWEs, the BBM-BBM equations, and the Svärd-Kalisch equations at four different points in time for $N = 512$ points in space (DOFs).
For the BBM-BBM equations and for the Svärd-Kalisch equations, we use periodic central finite
difference SBP operators of accuracy order $4$ with the semidiscretizations given by
\eqref{eq:semi-BBM-BBM1}--\eqref{eq:semi-BBM-BBM2} and \eqref{eq:semi-SK1}--\eqref{eq:semi-SK2},
respectively. Note that the initial condition is shifted around $0$ for the BBM-BBM equations
because we have only derived numerical methods for $\eta_0 = 0$. In the plots, the solution is
then shifted back. For the simulation of the SWEs we apply Trixi.jl
\cite{schlottkelakemper2020trixi,schlottkelakemper2021purely,ranocha2022adaptive}, which uses
the DGSEM with polynomials of degree 3 in space. As surface flux we take the energy-dissipative flux from \cite{audusse2004fast}
and as volume flux we use the numerical flux function presented in \cite{wintermeyer2017entropy}. Both the
Svärd-Kalisch and the BBM-BBM equations give similar results, while the SWEs fail to model the dispersive waves. In
\Cref{fig:waterheight-over-time-dingemans}, we use the parameter set~2 for the
Svärd-Kalisch equations, but qualitatively similar results are also obtained for the sets 3, 4, and 5.
Sets 3 and 4, however, produce slightly too fast waves because for the wavenumber $k\approx 0.84$
the dispersion relation of set~2 is closer to the one of the Euler equations, see
\Cref{fig:disp-rels}. In the following, we take a closer look at the numerical solution
of the Svärd-Kalisch equations, where we use set~2 if not stated otherwise.

\begin{figure}[htbp]
	\centering
	\includegraphics[width=0.9\textwidth]{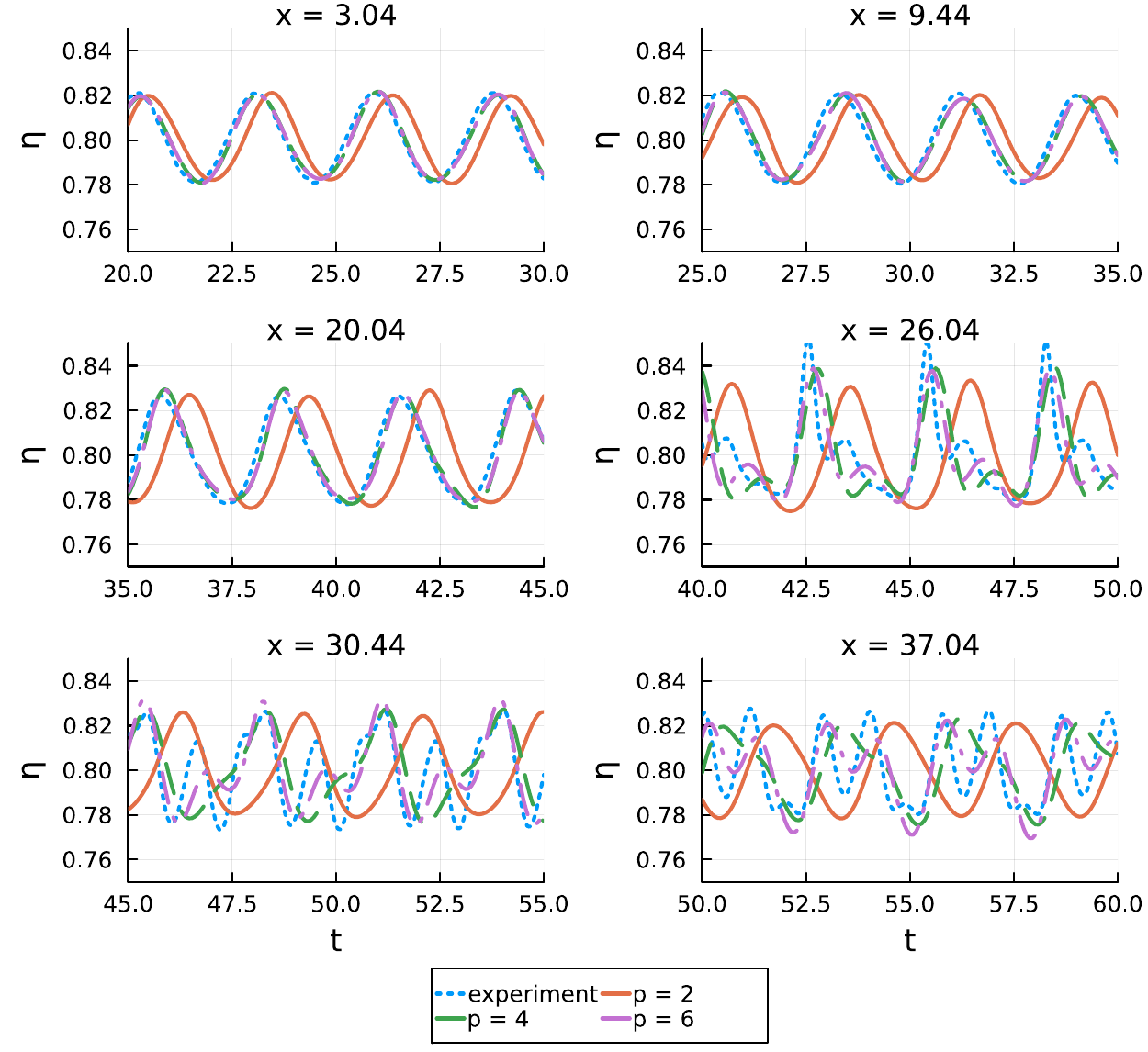}
	\caption{Water height at different locations for different orders of accuracy for the Svärd-Kalisch equations}
	\label{fig:waterheight-at-x-accuracy-orders-dingemans}
\end{figure}

During the experiment by Dingemans, measurements were taken at different positions indicated by the gray dashed lines in \Cref{fig:waterheight-over-time-dingemans}. The temporal development of the numerical solution of the Svärd-Kalisch equations at these points are shown in \Cref{fig:waterheight-at-x-accuracy-orders-dingemans} for different orders of accuracy $p$ with 512 nodes in space.
Comparing to the data obtained from the experiment, we can see that the numerical solution becomes much more precise by increasing the order of the discretization. While the authors of \cite{svard2023novel} use 36800 points in their finest mesh, the present discretization only needs 512 points to obtain similar good results if an order of accuracy of 4 or 6 is applied.

\begin{figure}[htbp]
    \centering
    \includegraphics[width=0.9\textwidth]{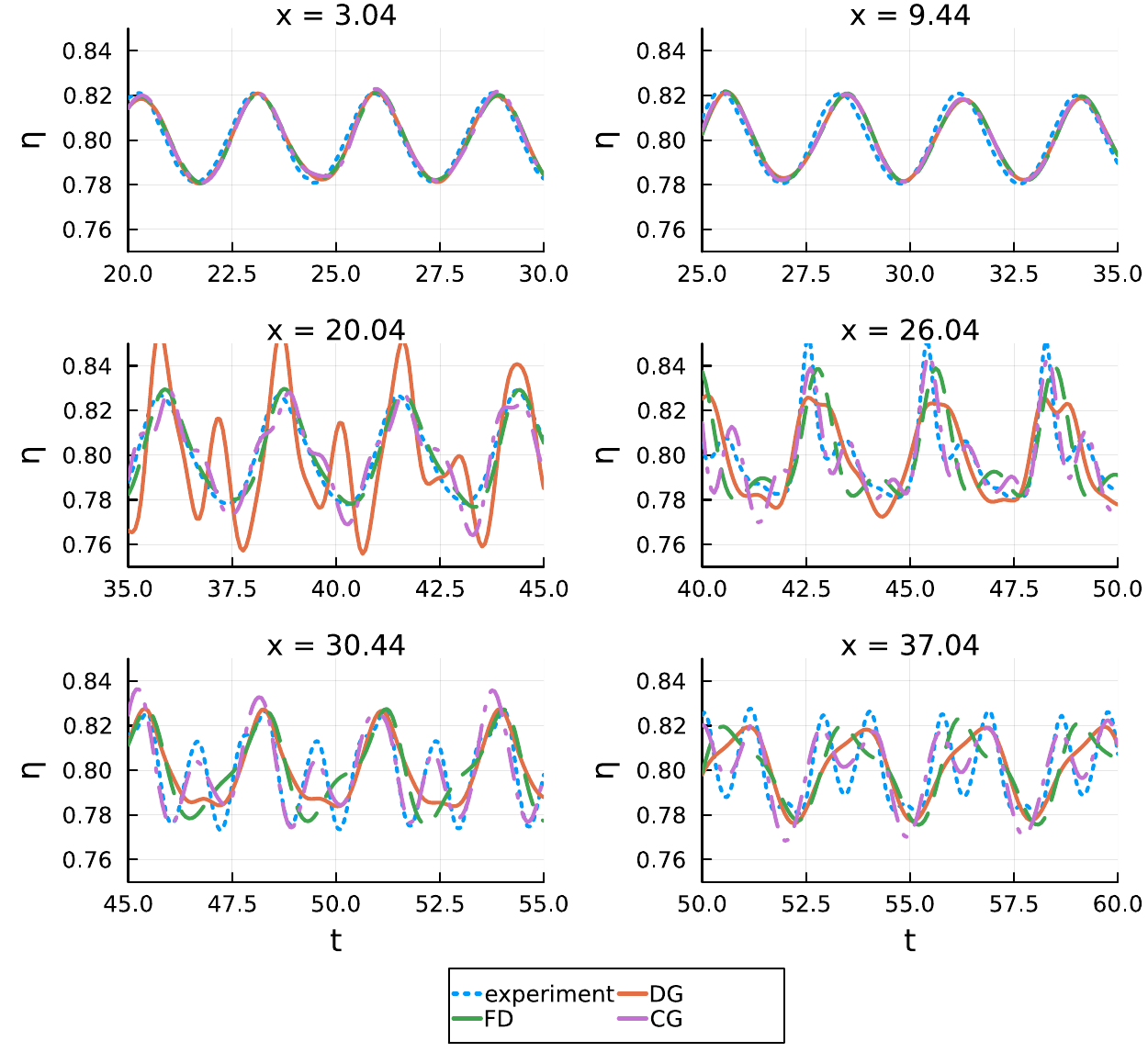}
    \caption{Water height at different locations for different solver types}
    \label{fig:waterheight-at-x-solver-types-dingemans}
\end{figure}

Next, we consider different solver types. The different results using narrow-stencil finite difference operators (FD) and coupled operators obtained from the nodal discontinuous Galerkin (DG) and the nodal continuous Galerkin (CG) method are shown in \Cref{fig:waterheight-at-x-solver-types-dingemans}.
All of the methods produce reasonable solutions. The CG method records the most details, which can especially be seen for the last two wave gauges, where the CG method is able to produce the smaller amplitude between two larger amplitudes for each wave.

\begin{figure}[htbp]
    \centering
    \includegraphics[width=0.9\textwidth]{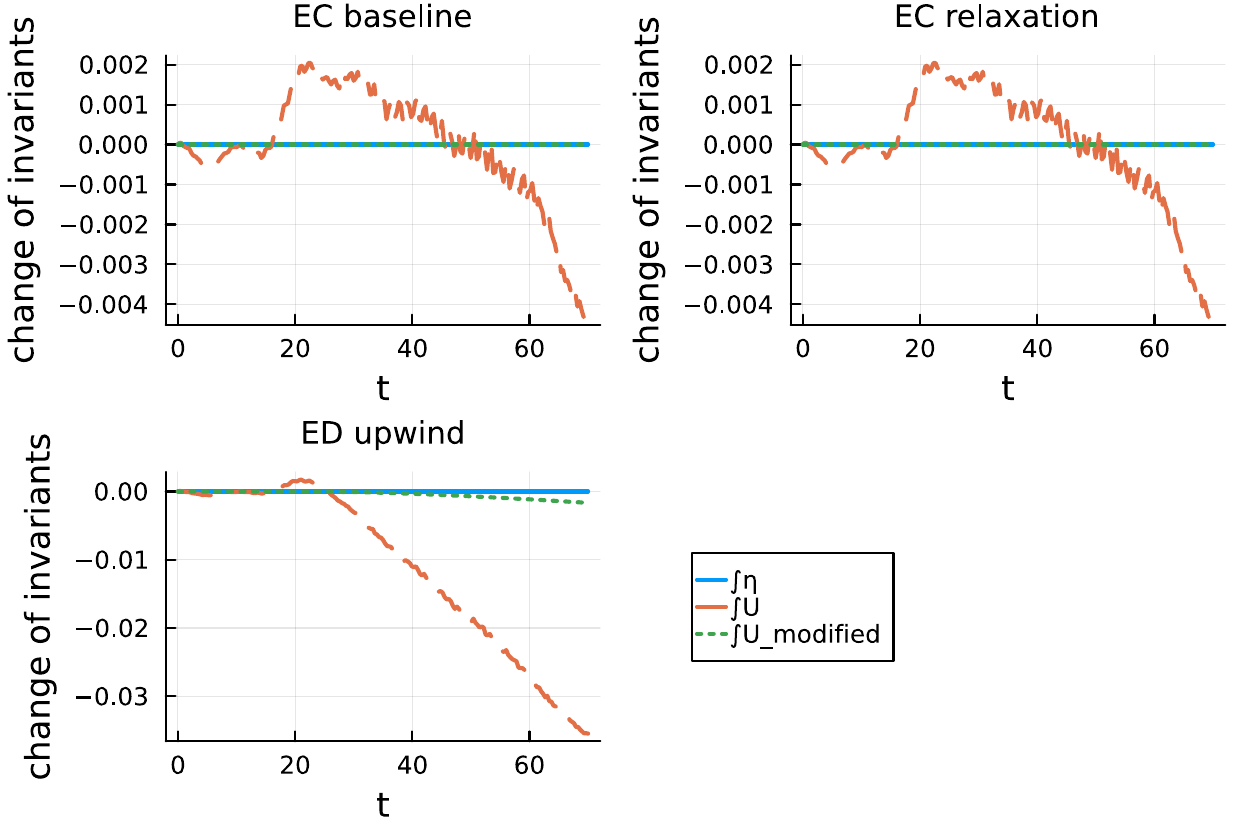}
    \caption{Development of linear and nonlinear invariants, top left: entropy-conservative (EC) scheme without RRK, top right: EC scheme with RRK, bottom left: entropy-dissipative (ED) scheme}
    \label{fig:invariants-ec-dingemans}
\end{figure}

Finally, we compare the two different semidiscretizations presented in \Cref{sect:num-svard-kalisch}. In particular, we compare the solutions of the semidiscretization \eqref{eq:semi-SK1}--\eqref{eq:semi-SK2} with finite difference operators without and with relaxation (applied to the total modified entropy) and the upwind semidiscretization \eqref{eq:semi-up-SK1}--\eqref{eq:semi-up-SK2} with finite difference operators. The first is entropy-conservative on the semidiscrete level, the second is entropy-conservative on the fully-discrete level, and the third is entropy-dissipative on the semidiscrete level, where the notion ``entropy-conservative (dissipative)'' refers to the total modified entropy. \Cref{fig:invariants-ec-dingemans} shows the total mass $\cM$, the total entropy $\cE$, and the total modified entropy $\hat\cE$ over time. All
solutions conserve the total mass exactly. The entropy-conservative semidiscretization without relaxation (upper left) already conserves the total modified entropy up to 7 digits. Using relaxation (upper right), the total modified entropy can be conserved exactly. In both cases, the total entropy $\cE$ changes up to 3 digits. For the (modified) entropy-dissipative semidiscretization of the Svärd-Kalisch equations based on upwind operators (bottom left), the dissipation is clearly observed in the modified total entropy and also in the usual total entropy.

\begin{figure}[htbp]
	\centering
	\includegraphics[width=0.9\textwidth]{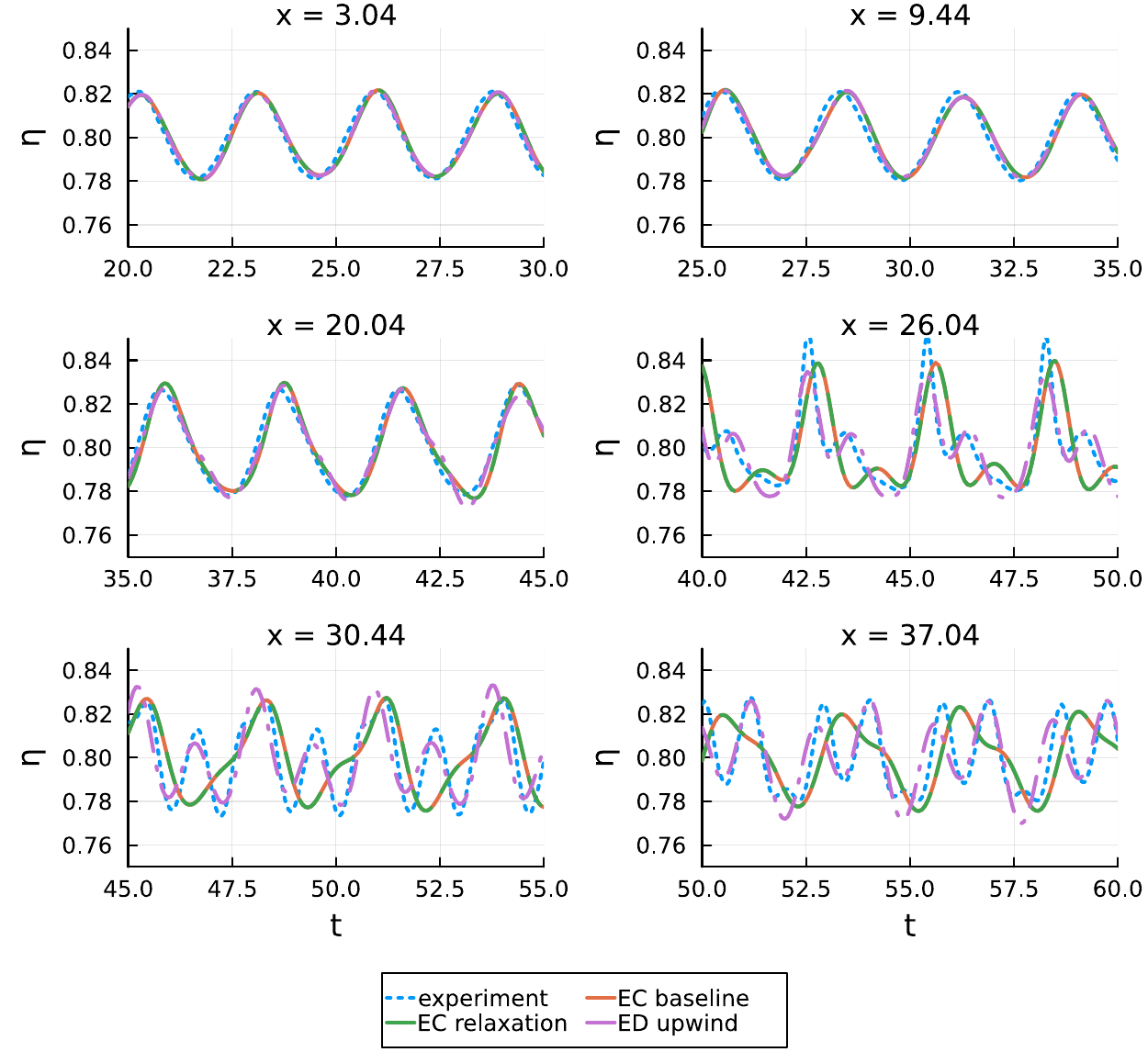}
	\caption{Water height at different locations for the entropy-conservative (EC) scheme with and without relaxation, and the entropy-dissipative (ED) scheme}
	\label{fig:waterheight-at-x-ec-dingemans}
\end{figure}

For completeness, we also compare the solutions at the different positions in space, which are depicted in \Cref{fig:waterheight-at-x-ec-dingemans}.
While the difference between the baseline and the relaxation entropy-conservative schemes is indistinguishable, the entropy-dissipative scheme is able to capture much more details. This indicates that, in order to obtain higher derivative SBP operators, it can be beneficial to replace consecutive central first-derivative SBP operators leading to wide-stencil operators by consecutive first-derivative upwind operators in different directions, which results in an operator with narrower stencil. By increasing the order of accuracy to $p = 6$ and the number of discretization points (DOFs) to $N = 1024$, we demonstrate in \Cref{fig:waterheight-at-x-order6-upwind-dingemans} that the solution obtained from the dissipative upwind scheme matches the experimental data remarkably well.
The figure also shows the solution using the same discretization, but applying the isotropic set~5. For this application set~2 performs better, but using set~2 also leads to reasonable solutions especially at the first 4 wave gauges.
\begin{figure}[ht]
	\centering
	\includegraphics[width=0.9\textwidth]{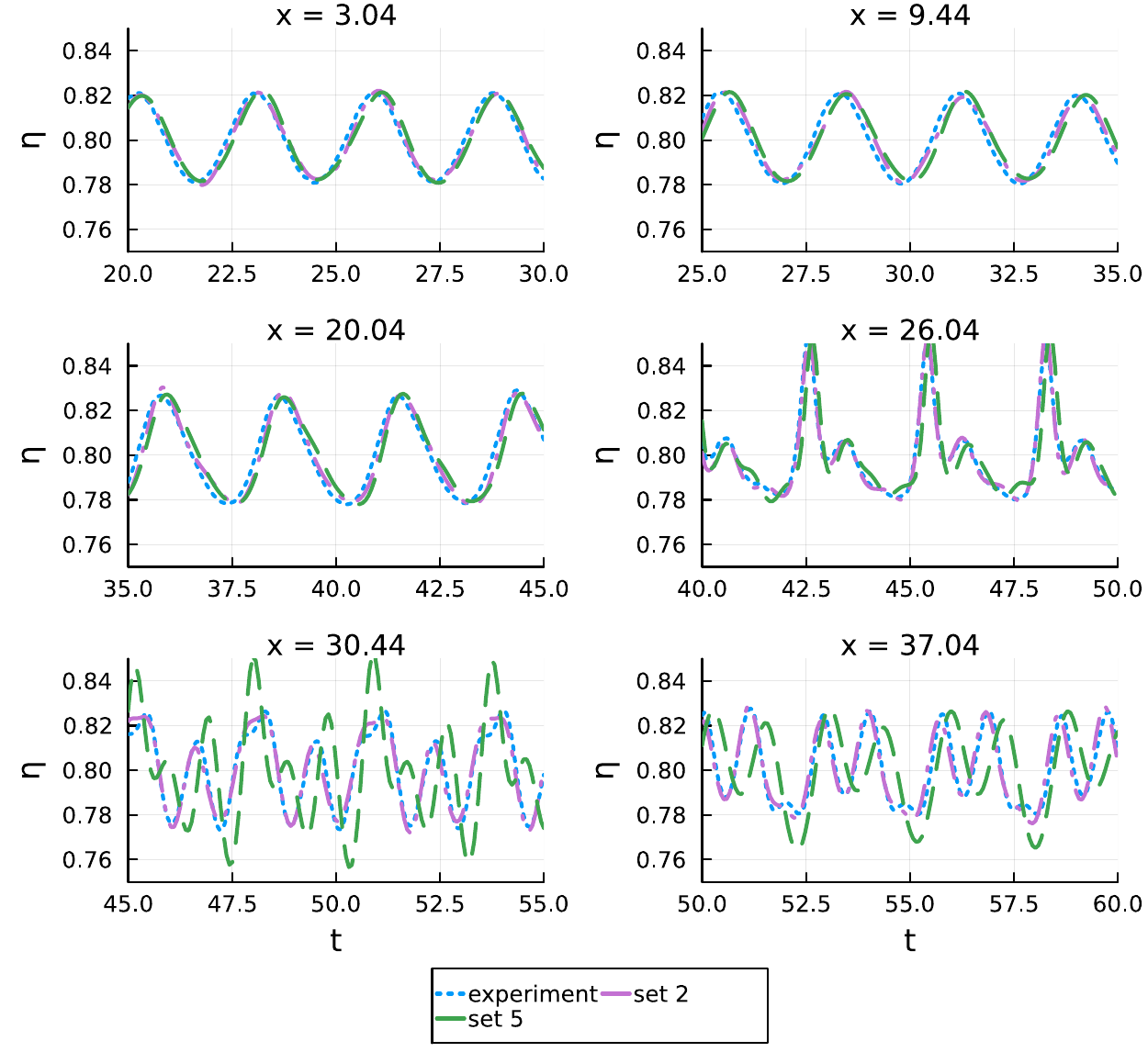}
	\caption{Water height at different locations for the entropy-dissipative scheme, $p = 6$, and $N = 1024$}
	\label{fig:waterheight-at-x-order6-upwind-dingemans}
\end{figure}

For comparison, the same setup is also solved with a semidiscretization that is neither entropy-conservative nor entropy-dissipative, e.g., using \eqref{eq:semi-SK1}--\eqref{eq:semi-SK2}, but without the split forms (i.e., replacing $\frac{1}{2}(D_1(\bm h\bm v^2) + \bm h\bm vD_1\bm v + \bm vD_1(\bm h\bm v))$ by $D_1(\bm h\bm v^2)$, and $\frac{1}{2}(D_1(\bm v\bm y) + \bm vD_1\bm y + \bm yD_1\bm v)$ by $D_1(\bm v\bm y)$ in \eqref{eq:semi-SK2}). In that case, the changes over time of the total entropy and the total modified entropy are considerably higher (up to 3 orders of magnitude for the modified entropy), as expected from the entropy analysis.

For completeness, we briefly show the results obtained by the solution of the BBM-BBM equations in \Cref{fig:waterheight-at-x-accuracy-orders-dingemans-BBM-BBM}.
\begin{figure}
	\centering
	\includegraphics[width=0.9\textwidth]{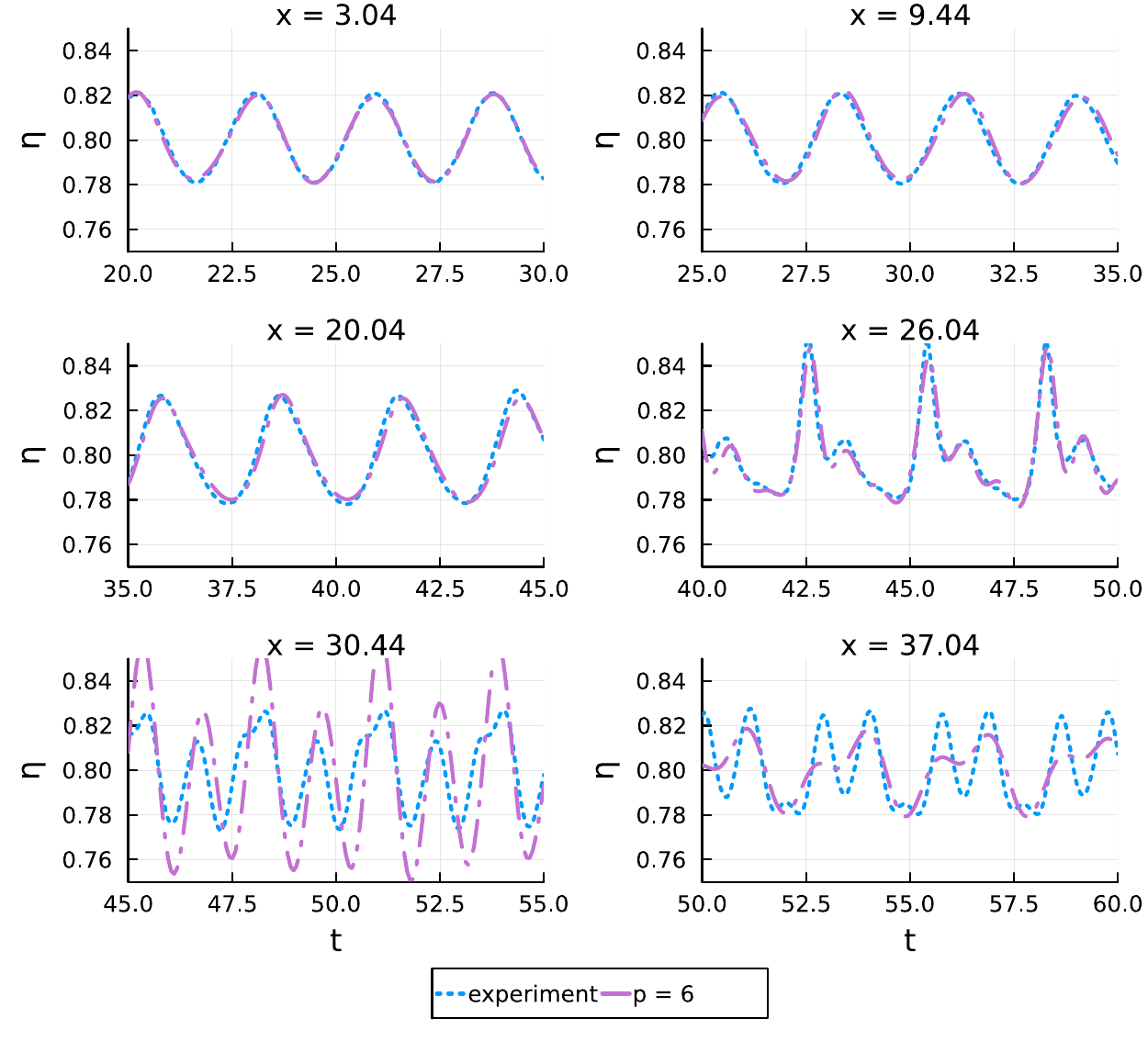}
	\caption{Water height at different locations for an order of accuracy $p = 6$ obtained by upwind discretization of the BBM-BBM equations with $N = 1024$}
	\label{fig:waterheight-at-x-accuracy-orders-dingemans-BBM-BBM}
\end{figure}
While the main features of the waves are captured by the solution, it contains less details than the solution of the Svärd-Kalisch equations, especially for the last two wave gauges.

\section{Summary and conclusions}\label{sect:conclusions}

We developed structure-preserving numerical methods for two nonlinear dispersive
systems modeling shallow water waves: the BBM-BBM equations
\cite{israwi2021regularized,israwi2023equations} and the model proposed recently
by Svärd and Kalisch \cite{svard2023novel,svard2025novel}. To obtain provably
energy-conserving or energy-stable schemes for periodic and reflecting wall
boundary conditions, we used summation-by-parts (SBP) operators in space and
relaxation methods in time.

We compared the two models in some numerical experiments including the experimental
setup of Dingemans \cite{dingemans1994comparison,dingemans1997water}. For these
small-amplitude waves, the (generalized) Svärd-Kalisch equations yield more accurate
results than the BBM-BBM equations as expected based on their linear dispersion relation.
In general, we have observed that applying upwind operators for higher-order derivatives
improves the approximations compared to using central finite difference operators.
The numeriacl results demonstrate the advantage of energy-conserving methods for long-term
simulations.

\section*{Acknowledgments}

JL acknowledges the support by the Deutsche Forschungsgemeinschaft (DFG)
within the Research Training Group GRK 2583 ``Modeling, Simulation and
Optimization of Fluid Dynamic Applications''.
HR was supported by the Deutsche Forschungsgemeinschaft
(DFG, German Research Foundation, project number 513301895)
and the Daimler und Benz Stiftung (Daimler and Benz foundation,
project number 32-10/22).

\printbibliography

\end{document}